\documentclass[11pt]{amsart}
\usepackage[margin=1in]{geometry}

\usepackage{amsthm}
\usepackage{mathrsfs}
\usepackage{mathtools}
\usepackage[all,arc]{xy}
\usepackage{tikz}
\usetikzlibrary{cd}
\usetikzlibrary{decorations.pathreplacing}
\usetikzlibrary{graphs,decorations.pathmorphing,decorations.markings}
\usepackage{tikz-cd}
\usepackage{faktor}
\usepackage{cancel}
\usepackage{slashed}
\usepackage{tensor}
\usepackage{tcolorbox}
\usepackage{enumerate}[shortlabel]
\usepackage{hyperref}
\hypersetup{%
	colorlinks=false,
    linkbordercolor=red,
    citebordercolor=green
}

\let\savedegree\degree
\let\degree\relax
\usepackage{mathabx}
\let\degree\savedegree

\newcommand\norm[1]{\left\lVert#1\right\rVert}
\newcommand\normc[1]{\left\lVert#1\right\rVert_\mathbb{C}}
\DeclareMathOperator{\modulus}{mod}

\DeclarePairedDelimiter{\abs}{\lvert}{\rvert}

\theoremstyle{plain}
\newtheorem{thm}{Theorem}[section]

\newtheorem{lemma}[thm]{Lemma}

\newtheorem{prop}[thm]{Proposition}

\theoremstyle{definition}
\newtheorem{remark}[thm]{Remark}

\newtheorem{defn}[thm]{Definition}

\newtheorem{example}[thm]{Example}

\newtheorem*{condition}{Mass Condition}

\newenvironment{proofoftheorem1.1}{%
  \proof}{\endproof}

\begin{document}

\title{A partial characterization of cosine Thurston maps}

\author{Schinella D'Souza}
\address{Department of Mathematics, University of Michigan, Ann Arbor, Michigan 48109}
\email{dsouzas@umich.edu}

\thanks{The author was supported in part by NSERC}

\begin{abstract}
In this paper, we introduce cosine Thurston maps. In particular, we construct postsingularly finite topological cosine maps and focus on such maps with strictly preperiodic critical points. We use the techniques of Hubbard, Schleicher, and Shishikura to prove that, subject to a condition on the critical points,  a postsingularly finite topological cosine map with strictly preperiodic critical points is combinatorially equivalent to $C_\lambda(z) = \lambda\cos z$ for a unique $\lambda \in \mathbb{C}^*$ if only if it has no degenerate Levy cycle. 
\end{abstract}

\maketitle

\section{Introduction}
\label{introduction_section}

William Thurston's topological characterization of rational functions (\cite{douady_hubbard_thurston}) is a cornerstone theorem in rational dynamics. The framework Thurston built between topology and geometry in dynamics involves \textit{postcritically finite} maps (maps where all critical points have orbits that are periodic or preperiodic). 

A natural question that has since arisen is how to extend the theory Thurston built to transcendental maps, where the analog of postcritically finite maps are \textit{postsingularly finite} maps. The characterization does not go through automatically because it is crucial that the maps in question have finite degree. Furthermore, entire transcendental functions have an essential singularity at $\infty$. However, a major breakthrough in this direction was due to Hubbard, Schleicher, and Shishikura (\cite{expthursmaps}). In their work, the authors expand Thurston theory to the exponential family $z \mapsto \lambda e^z$ and prove a topological characterization of exponential maps. Their results are based on understanding the geometry of quadratic differentials on the Riemann sphere $\hat{\mathbb{C}}$ and then applying it to exponential maps. Their work on controlling the geometry of quadratic differentials leaves room for application to other maps as well.

One can then ask the same question about cosine maps, which are understood less well than exponential maps. In the literature, the cosine family often looked at is the two-parameter family $C_{a,b}: z \mapsto ae^z + be^{-z}$ for $a, b \in \mathbb{C}^*$. A classification of the escaping points of this family was provided by \cite{schleicher_escaping_pts_cosine}. Furthermore, Schleicher studied the dynamics of $C_{a,b}$ in \cite{schleicher_fine_structure_cosine} in the case where the critical orbits are strictly preperiodic. We will be interested in this case as well, but for not exactly the same cosine family. From a topological point of view in the dynamical plane, \cite{leticia_cosine} has given topological models for subclasses of $C_{a,b}$. As for one-parameter families, the core entropy of the cosine family $z \mapsto \frac{\lambda}{2}(e^z+e^{-z})$ has been proven to be bounded dependent on combinatorial conditions in \cite{chernov_core_entropy_cosine}.

In this paper, we consider the same question of \cite{expthursmaps} but for the one-parameter cosine family $C_\lambda(z) = \lambda \cos(z)$ (where $\lambda \in \mathbb{C}^*$), thought of on $\mathbb{C}$ due to the essential singularity at $\infty$. We focus on the case where all critical points have orbits that are strictly preperiodic. There are many merits to this family that make it particularly amenable to adaptations of the results of Hubbard, Schleicher, and Shishikura. Their tools have also been used in \cite{shemyakov} to study certain postsingularly finite entire functions (such as \textit{structurally finite maps}). In the cosine family $C_\lambda$, the obstacle is the presence of infinitely many critical points and therefore, we impose an additional constraint on the position of critical points relative to points in the postsingular set. We also provide an example that motivates this constraint. 

\subsection{The one-parameter cosine family}

Let $\lambda \in \mathbb{C}^*$. The holomorphic cosine map we work with is defined by
\begin{align*}
    C_\lambda: \mathbb{C} &\rightarrow \mathbb{C} \\
    z &\mapsto \lambda\cos z . 
\end{align*}

There are infinitely many critical points of $C_\lambda$ and they are exactly at $k\pi$ for $k \in \mathbb{Z}$. They are completely independent of the critical values, $\lambda$ and $-\lambda$. Each critical point of $C_\lambda$ has the unique property that it maps to its critical value with local degree 2. Because $C_\lambda$ is an even map, $\lambda$ and $-\lambda$ have the same image and therefore all the critical orbits merge. 

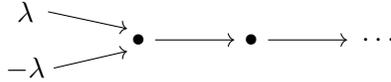
\begin{center}
    \begin{tikzpicture}[scale=1.5]
        \draw (-1,0.25) node{$\lambda$};
        \draw (-1,-0.25) node{$-\lambda$};
        \draw (0,0) node{$\bullet$};
        \draw (1,0) node{$\bullet$};
        \draw (2.15,0) node{$\cdots$};
        \draw[->] (-0.8,0.25) to (-0.1,0.1);
        \draw[->] (-0.75,-0.25) to (-0.1,-0.1);
        \draw[->] (0.15,0) to (0.85,0);
        \draw[->] (1.15,0) to (1.85,0);
    \end{tikzpicture} 
\end{center}

\subsection{Adaptation of exponential results to cosine}

This work applies techniques of \cite{expthursmaps} to build a partial characterization of postsingularly finite topological cosine maps (see Definition \ref{topcosmap}) through the lens of Thurston theory. Due to the parallel of this paper with the work of Hubbard, Schleicher, and Shishikura, we use their notation so that the reader may move between both papers easily and use them to expand this work to further transcendental functions.

\subsubsection{Main result} We adapt key notions of Thurston theory including postcritically finite topological maps, combinatorial equivalence, and degenerate Levy cycles (see Section \ref{thurston_theory_section}). We build a topological cosine map (see Definition \ref{topcosmap}) that is postsingularly finite (see Definition \ref{topcospsf}). The notion of postsingularly finite and postcritically finite are the same on $\mathbb{C}$ for cosine because, like a rational map, it has only critical points and critical values (and no asymptotic values). By adapting the techniques of \cite{expthursmaps}, we are able to recover the following main result of our work. 

\begin{thm}[Partial characterization of cosine maps - preperiodic case]\label{main_thm}
    Let $f$ be a postsingularly finite topological cosine map with strictly preperiodic critical points and suppose the \hyperref[condition:mass_condition]{mass condition} holds. Then $f$ is combinatorially equivalent to a unique postsingularly finite holomorphic cosine map if and only if it does not admit a degenerate Levy cycle. 
\end{thm}

We note that the general proof of this main result is similar to the characterization of exponential maps in \cite{expthursmaps}. However, there are several propositions that go into the proof of the characterization of exponential maps that rely on the explicit properties of the exponential. These are the propositions that we adapt to the cosine setting. Furthermore, we provide two examples of candidate efficient sequences of quadratic differentials (in the exponential or cosine sense). The second example is motivation for the \hyperref[condition:mass_condition]{mass condition} in Theorem \ref{main_thm}, a condition that we impose to ensure we have enough control over the geometry of quadratic differentials on $\hat{\mathbb{C}}$.

\subsection*{Outline} Sections \ref{thurston_theory_section} and \ref{teich_quad_spaces_section} provides the background definitions necessary to understand Theorem \ref{main_thm}. In Section \ref{thurston_theory_section}, we give Teichm\"uller space background and construct an analytic map from a specific Teichm\"uller space to itself (the Thurston pullback map). In Section \ref{teich_quad_spaces_section}, we give background on quadratic differentials, their relationship to Teichm\"uller space, and specify the \hyperref[condition:mass_condition]{mass condition}. We also provide brief proof sketches of the results of \cite{expthursmaps} that generalize immediately to cosine (including Theorem \ref{main_thm}), mainly highlighting key differences as well as situations when the \hyperref[condition:mass_condition]{mass condition} need not be used. In Section \ref{cos_eff_quad_diff_section}, the notion of cosine efficiency for a sequence of quadratic differentials is defined and two examples of such candidate sequences is given. In Section \ref{adaptations_to_cosine_section}, we adapt the exponential-dependent results of \cite{expthursmaps} to the cosine setting. Lastly, in Section \ref{generalizing_results_further_section}, we explain what kinds of properties would be needed to generalize the results of Hubbard, Schleicher, and Shishikura to further transcendental maps and discuss weakening the imposed \hyperref[condition:mass_condition]{mass condition}.

\subsection*{Acknowledgements}
The author would like to express her immense gratitude to John Hubbard for suggesting this work and for numerous vital discussions regarding the techniques of the characterization of exponential maps. The author is very thankful to Sarah Koch for her constant support, and especially for providing many creative and perceptive comments. The author also thanks Max Lahn and Malavika Mukundan for several valuable discussions. This work was supported in part by NSERC.

\section{Thurston Theory for Cosine Maps}
\label{thurston_theory_section}

In this section, we will set up analogs of Thurston mappings and the Thurston pullback map for cosine. We will work with the topological sphere $S^2$ and the Riemann sphere $\hat{\mathbb{C}}$, with finitely many punctured points. We will assume all maps are orientation-preserving and there is a point $\infty$ on $S^2$ (by fixing an identification of $S^2$ with $\hat{\mathbb{C}}$ so that the north pole of $S^2$ is identified with $\infty \in \hat{\mathbb{C}}$).

\begin{defn}\label{topcosmap}
    A \textit{topological cosine map} is a ramified covering map $f:S^2 \setminus \{ \infty\} \rightarrow S^2 \setminus \{\infty\}$ that has two critical values, denoted $x_1$ and $y_1$, that are mapped together under $f$, along with an embedding of $\mathbb{Z}$ into $S^2 \setminus \{\infty \}$ such that if $b \in S^2 \setminus \{\infty, x_1, y_1\}$ and $\pi_1(S^2 \setminus \{\infty, x_1, y_1\}, b) := \langle \alpha, \beta \rangle$, then for every $n \in \mathbb{Z}$
    \begin{enumerate}[(1)]
        \item $f^{-1}(\alpha)$ consists of paths that connect $2n$ to $2n+1$ twice and
        \item $f^{-1}(\beta)$ consists of paths that connect $2n+1$ to $2n+2$ twice.
    \end{enumerate}
\end{defn}

\tikzset{->-/.style={decoration={
  markings,
  mark=at position .5 with {\arrow{>}}},postaction={decorate}}}

\begin{figure}[!b]
\centering
\begin{tikzpicture}[ scale = 1.7 , every loop/.style={}]

    \draw[thick, ->] (0,-0.5) -- (0,-1.5) node[midway, right]{$f$};

    \draw[thick, ->-, red] (-1,0) to [out=90, in=90] (-2,0);
    \draw[thick, ->-, red] (-2,0) to [out=-90, in=-90] (-1,0);
    \draw[thick, ->-, blue] (0,0) to [out=90, in=90] (-1,0);
    \draw[thick, ->-, blue] (-1,0) to [out=-90, in=-90] (0,0);
    \draw[thick, ->-, red] (1,0) to [out=90, in=90] (0,0);
    \draw[thick, ->-, red] (0,0) to [out=-90, in=-90] (1,0);
    \draw[thick, ->-, blue] (2,0) to [out=90, in=90] (1,0);
    \draw[thick, ->-, blue] (1,0) to [out=-90, in=-90] (2,0);
    \draw[thick, ->-, red] (2.5,0.3) to [out=170, in=90] (2,0);
    \draw[thick, ->-, red] (2,0) to [out=-90, in=-170] (2.5,-0.3);
    \draw[thick, ->-, blue] (-2,0) to [out=90, in=10] (-2.5,0.3);
    \draw[thick, ->-, blue] (-2.5,-0.3) to [out=-10, in=-90] (-2,0);

    \draw (-2,0) node{$\bullet$} node[right]{$2n-2$};
    \draw (-1,0) node{$\bullet$} node[right]{$2n-1$};
    \draw (0,0) node{$\bullet$} node[right]{$2n$};
    \draw (1,0) node{$\bullet$} node[right]{$2n+1$};
    \draw (2,0) node{$\bullet$} node[right]{$2n+2$};
    \draw (3,0) node{$\cdots$};
    \draw (-2.8,0) node{$\cdots$};

    \draw[scale=3, thick, ->-, red] (0,-0.7) to [out=-20, in=60, loop] (0,-1.2);
    \draw[scale=3] (0.35,-0.6) node[red]{$\alpha$};
    \draw[scale=3, thick, ->-, blue] (0,-0.7) to [out=160, in=240, loop] (0,-0.7);
    \draw[scale=3] (-0.35,-0.8) node[blue]{$\beta$};
    \draw[scale=3] (0,-0.7) node{$\bullet$} node[right,below]{$b$};
    \draw[scale=3] (0.1,-0.65) node{$\bullet$} node[right]{$x_1$};
    \draw[scale=3] (-0.2,-0.75) node{$\bullet$} node[right]{$y_1$};
\end{tikzpicture}
\caption{A topological cosine map $f:S^2 \setminus \{\infty\} \rightarrow S^2 \setminus \{\infty\}$.}
\label{top_cos_map_fig}
\end{figure}
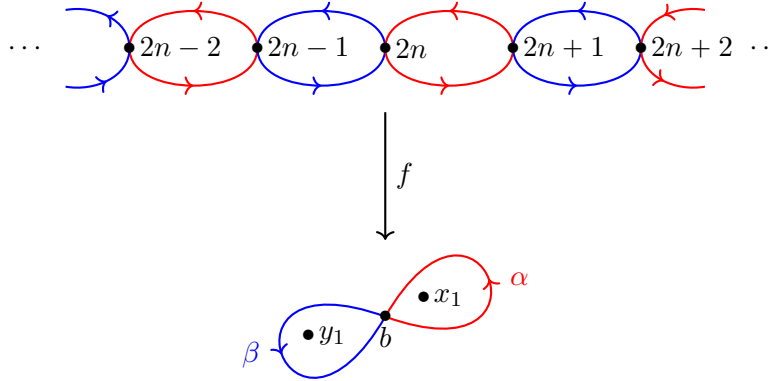

Figure \ref{top_cos_map_fig} illustrates how one can picture a topological cosine  map. Observe that $f$ is an infinite-degree cover with every critical point mapping forward with local degree 2. This general idea is similar in spirit to recent work of Mukundan, Prochorov, and Reinke with admissible quadruples (\cite{mukundan_prochorov_reinke}).

\begin{defn}\label{topcospsf}
    Given a topological cosine map $f$, the \textit{postsingular set of $f$} is defined as 
    \[
    P_f := \bigcup_{n\geq 0}f^{\circ n}(x_1) \cup \{y_1, \infty\}.
    \]
    where $x_1$ and $y_1$ are the two critical values of $f$. A topological cosine map is called \textit{postsingularly finite} if $\left|P_f \right| < \infty$.
\end{defn}

\begin{defn}\label{topcosthurstonequiv}
    Let $f, g: S^2 \setminus \{\infty\} \rightarrow S^2 \setminus \{ \infty \}$ be topological cosine maps. We say that $f$ and $g$ are \textit{combinatorially equivalent} if there exists homeomorphisms $\varphi, \varphi': S^2 \rightarrow S^2$ such that $\varphi|_{P_f} = \varphi'|_{P_f}$, $P_g = \varphi(P_f) = \varphi'(P_f)$, $\varphi(\infty) = \varphi'(\infty) = \infty$, $\varphi$ and $\varphi'$ are isotopic rel $P_f$, and the diagram
    \begin{center}
    \begin{tikzcd}
    S^2 \setminus\{\infty \} \arrow[d, "f"'] \arrow[r, "\varphi'"] & S^2 \setminus\{\infty \} \arrow[d, "g"] \\
    S^2 \setminus\{\infty \} \arrow[r, "\varphi"]                  & S^2 \setminus\{\infty \}               
    \end{tikzcd}
    \end{center}
    commutes.
\end{defn}

\begin{defn}\label{essential_def}
    Given a finite set $A \subset S^2$, a simple closed curve $\gamma \subset S^2 \setminus A$ is said to be \textit{essential} if each component of $S^2 \setminus \gamma$ contains at least two points of $A$.
\end{defn}

\begin{defn}\label{Levy_cycle}
    A topological cosine map $f$ has a \textit{Levy cycle} if there exists a circularly ordered set of disjoint, simple closed curves $\Gamma = \{\gamma_0, \ldots, \gamma_{k-1}, \gamma_k = \gamma_0 \}$ on $S^2 \setminus P_f$ with the following property: for each $\gamma_i \in \Gamma$, $\gamma_i$ is essential, there exists a component $\gamma_{i-1}' \in f^{-1}(\gamma_i)$ that is homotopic to $\gamma_{i-1}$ rel $P_f$, and $f$ is a homeomorphism between $\gamma_{i-1}'$ and $\gamma_i$. If $D_{\gamma_i}$ denotes the topological disk $\gamma_i$ bounds not containing $\infty$, then if $f$ is a homeomorphism between $\overline{D}_{\gamma_{i-1}}$ and $\overline{D}_{\gamma_{i}}$ for $0 \leq i \leq k-1$, then $\Gamma$ is called a \textit{degenerate Levy cycle}.
\end{defn}

In Definition \ref{Levy_cycle}, components that do not contain $\infty$ in $S^2$ will be called \textit{bounded}.

\subsection{Teichm\"uller space background and notation}

Let $f$ be a postsingularly finite topological cosine map. The following definition is equivalent to the standard Teichm\"uller space definition (such as in \cite{hubbardvol1}). This definition will serve more useful for our purposes. 

\begin{defn}
    Let $\varphi_1, \varphi_2: S^2 \rightarrow \hat{\mathbb{C}}$ be homeomorphisms. Then, an element of the \textit{Teichm\"uller space modeled on $(S^2, P_f)$}, denoted $\mathcal{T}_f$, is defined by an equivalence class $[\varphi_1]$ such that we have $\varphi_1(x_1) = -\varphi_1(y_1)$ and $\varphi_1(\infty) = \infty$ with $\varphi_2 \in [\varphi_1]$ if $\varphi_1|_{P_f} = \varphi_2|_{P_f}$ and $\varphi_1$ is isotopic to $\varphi_2$ relative to $P_f$.
\end{defn}

If $|P_f| < \infty$, then write $P_f = \{x_1, y_1, x_2, x_3, \ldots, x_k, \infty \}$, where $x_{i+1} = f^{\circ i}(x_1)$ for $1 \leq i \leq k-1$. Then, $\mathcal{T}_f$ is finite-dimensional of dimension $|P_f| - 3$ and our results assume $|P_f| > 3$. For $f$ with strictly preperiodic critical points, the preperiod refers to the preperiod of the orbit of any critical point of $f$ (it will be the same regardless of the critical point chosen).

\begin{example}
    Suppose $f$ is a postsingularly finite topological cosine map $f$ with $|P_f| = 4$ such that $f$ has preperiod 2 and period 1. Write $P_f = \{x_1, y_1, x_2, \infty \}$ and note the dynamical system looks like
    \begin{center}
        \begin{tikzpicture}[scale=1.5]
            \draw (-1,0.25) node{$x_1$};
            \draw (-1,-0.25) node{$y_1$};
            \draw (0.1,0) node{$x_2$};
            \draw[->] (-0.8,0.25) to (-0.1,0.1);
            \draw[->] (-0.75,-0.25) to (-0.1,-0.1);
            \draw[->] (0.25,0.1) to [out=40,in=320,looseness=6] (0.25,-0.1);

            \draw (2,0) node{$\infty$};
            \draw[->] (2.15,0.1) to [out=40,in=320,looseness=6] (2.15,-0.1);
        \end{tikzpicture} 
    \end{center}
\end{example}

\subsection{Construction of Thurston pullback map for cosine}

Having set up the topological map associated to cosine, our next goal will be to define an analog of the Thurston pullback map for it that maps $\mathcal{T}_f$ to itself. 

Let $[\varphi] \in \mathcal{T}_f$. Suppose $\varphi(x_1) \neq 0$, $\lambda := \varphi(x_1)$, and $\varphi(y_1) = -\lambda$. Then, $\varphi$ is unique up to scaling. Choose a basepoint $b \in S^2 \setminus \{x_1,y_1, \infty \}$. From Definition \ref{topcosmap}, we can assume, without loss of generality, that $\alpha$ is the generator of $\pi_1(S^2 \setminus \{\infty\},b)$ that is a loop around $x_1$. Denote $\pi_1(\mathbb{C}, 0) := \langle \gamma, \delta \rangle$ and assume, without loss of generality, that $\gamma$ is a loop around $\lambda$. For a loop $\eta$, we will denote $D_\eta$ to be the topological disk consisting of the loop $\eta$ along with its bounded component. 
    
Now $f^{-1}(D_{\alpha^2})$ consists of infinitely many disjoint topological disks. The restriction of $\varphi \circ f$ to any one of these topological disks is a double cover from that topological disk to an open bounded subset of $\mathbb{C}$ ramified at $\lambda$ which, in this case, is $D_\gamma$. Similarly, $C_\lambda^{-1}(D_{\gamma^2})$ also consists of infinitely many disjoint topological disks, each of which contains exactly one critical point of $C_\lambda$. The restriction of $C_\lambda$ to any one of these topological disks is a double cover from that topological disk to $D_\gamma$, ramified at $\lambda$.

To begin the construction, choose a disk $D \in f^{-1}(D_{\alpha^2}) \subset S^2 \setminus \{\infty \}$ and $D' \in C_\lambda^{-1}(D_{\gamma^2}) \subset \mathbb{C}$. Then, $\varphi \circ f|_D:D \rightarrow D_{\gamma}$ and $C_\lambda|_{D'}: D' \rightarrow D_\gamma$ are both double covers ramified at a single point. There exists only one such double cover up to sign so there exists a lift $\varphi':D \rightarrow D'$ that is a homeomorphism onto its image. It is unique up to translation by $2\pi$ and sign because $C_\lambda(z) = C_\lambda(z \pm 2n\pi)$. See Figure \ref{Thurston_pullback_map_fig} for a pictorial description of this initial step. 
    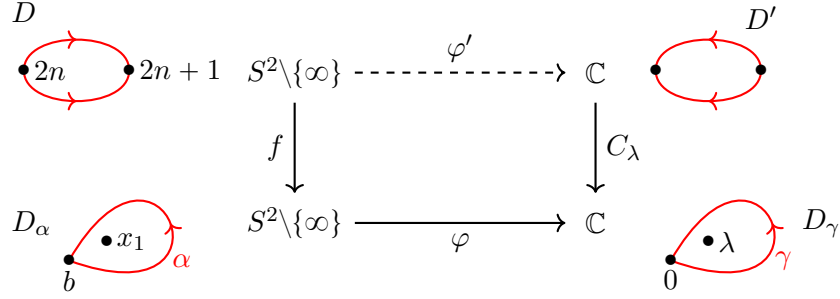
\begin{figure}
    \centering
    \begin{tikzpicture}[ scale = 2 , every loop/.style={}]
        \draw (-1,1) node{$S^2 \setminus \{\infty \}$};
        \draw (1,1) node{$\mathbb{C}$};
        \draw (-1,0) node{$S^2 \setminus \{\infty \}$};
        \draw (1,0) node{$\mathbb{C}$};
        \draw[thick, ->] (-1,0.8) -- (-1,0.2) node[midway, left]{$f$};
        \draw[thick, ->] (1,0.8) -- (1,0.2) node[midway, right]{$C_\lambda$};
        \draw[thick, ->] (-0.6,0) -- (0.8,0) node[below, midway]{$\varphi$};
        \draw[thick, dashed, ->] (-0.6,1) -- (0.8,1) node[above, midway]{$\varphi'$};
    
        \draw[scale=0.7, thick, ->-, red] (-4,1.45) to [out=90, in=90] (-3,1.45);
        \draw[scale=0.7,thick, ->-, red] (-4,1.45) to [out=-90, in=-90] (-3,1.45);
        \draw[scale=0.7] (-4,1.45) node{$\bullet$} node[right]{$2n$};
        \draw[scale=0.7] (-3,1.45) node{$\bullet$} node[right]{$2n+1$};
        \draw[scale=0.7] (-4,2) node{$D$};
    
        \draw[scale=2.5, thick, ->-, red] (-1,-0.1) to [out=-20, in=60, loop] (-1,-1.2);
        \draw[scale=2.5] (-0.7,-0.1) node[red]{$\alpha$};
        \draw[scale=2.5] (-1,-0.1) node{$\bullet$} node[right,below]{$b$};
        \draw[scale=2.5] (-0.9,-0.05) node{$\bullet$} node[right]{$x_1$};
        \draw[scale=2.5] (-1.1,0) node{$D_\alpha$};
    
        \draw[scale=0.7, thick, ->-, red] (3,1.45) to [out=90, in=90] (2,1.45);
        \draw[scale=0.7,thick, ->-, red] (3,1.45) to [out=-90, in=-90] (2,1.45);
        \draw[scale=0.7] (3,1.45) node{$\bullet$};
        \draw[scale=0.7] (2,1.45) node{$\bullet$};
        \draw[scale=0.7] (3,1.95) node{$D'$};
    
        \draw[scale=2.5, thick, ->-, red] (0.6,-0.1) to [out=-20, in=60, loop] (0.75,-1.2);
        \draw[scale=2.5] (0.9,-0.1) node[red]{$\gamma$};
        \draw[scale=2.5] (0.6,-0.1) node{$\bullet$} node[right,below]{$0$};
        \draw[scale=2.5] (0.7,-0.05) node{$\bullet$} node[right]{$\lambda$};
        \draw[scale=2.5] (1,0) node{$D_\gamma$};
        
    \end{tikzpicture}
    \caption{Initial step in the construction of the Thurston pullback map for $C_\lambda(z) = \lambda \cos z$.}
    \label{Thurston_pullback_map_fig}
    \end{figure}

Next, note that $D$ intersects two topological disks in $f^{-1}(D_{\beta^2})$ at two points labeled by integers, according to the definition of a topological cosine map (Definition \ref{topcosmap}). Without loss of generality, choose the topological disk that contains the point labeled by the larger of these two integers, say $2n+1$, and call it $\widetilde{D}$. Similarly, choose the topological disk $\widetilde{D'} \in C_\lambda^{-1}(D_{\delta^2})$ that intersects $D'$ and contains the larger critical point of $C_\lambda$, say $m\pi$. As above, there exists a lift $\varphi'': \widetilde{D} \rightarrow \widetilde{D'}$ that is a homeomorphism onto its image such that $\varphi''(2n+1)=m\pi - \frac{\pi}{2}$. Therefore, we have obtained a homeomorphism from $D \cup \widetilde{D}$ to $D' \cup \widetilde{D'}$. For ease of notation, we call this homeomorphism $\varphi': D \cup \widetilde{D} \to D' \cup \widetilde{D'}$. Continue to repeat this procedure for every disk in $f^{-1}(D_{\alpha^2}) \cup f^{-1}(D_{\beta^2})$ to obtain a homeomorphism which we denote, for ease of notation, $\varphi': f^{-1}(D_\alpha \cup D_\beta) \rightarrow C_\lambda^{-1}(D_\gamma \cup D_\delta)$. 

At this point, we have the two remaining unbounded connected components of the complement of $f^{-1}(D_\alpha \cup D_\beta)$ to deal with in order to extend $\varphi'$ to a homeomorphism from $S^2 \setminus \{\infty \}$ to $\mathbb{C}$. For each unbounded component, $\varphi \circ f$ is a universal cover from the unbounded component to $\mathbb{C} \setminus (D_\gamma \cup D_\delta)$ and the same is true for $C_\lambda$ restricted to each unbounded component of $C_\lambda^{-1}(D_\gamma \cup D_\delta)$. Therefore, $\varphi'$ can be extended to a homemorphism $\varphi': S^2 \setminus \{\infty\} \rightarrow \mathbb{C}$ and in particular, it extends so that $\varphi'(\infty) = \infty$. By construction, the diagram   
\begin{center}
    \begin{tikzcd}
        S^2 \setminus\{\infty \} \arrow[d, "f"'] \arrow[r, "\varphi'"] & \mathbb{C} \arrow[d, "C_\lambda"] \\
        S^2 \setminus\{\infty \} \arrow[r, "\varphi"]                  & \mathbb{C}           
    \end{tikzcd}
\end{center}
commutes. Consider $[\varphi],  [\varphi'] \in \mathcal{T}_f$ and note that if $\psi$ is another representative of $[\varphi]$, then the resulting map $\psi'$ under the above construction will be in $[\psi]$ by the same logic as [\cite{epstein-remple-gillen}, Proposition 2.3]. Thus, the map defined by $\sigma_f: \mathcal{T}_f \rightarrow \mathcal{T}_f$ defined by $\sigma_f([\varphi]) = [\varphi']$ is well-defined and $C_\lambda$ is unique up to conjugation by affine maps. It is called the \textit{Thurston pullback map of $f$}.

A more general construction of the Thurston pullback is done for entire maps in \cite{mukundan_prochorov_reinke}. In contrast to this construction and the general rational construction (see [\cite{hubbardvol2}, Definition 10.6.1]), our definition of the Thurston pullback map for a topological cosine map relies on using the map $C_\lambda$ to provide the homeomorphism $\varphi'$ that completes the diagram in Figure \ref{Thurston_pullback_map_fig}. We know from the start that $C_\lambda$ is analytic and need only ensure $\varphi'$ has the required properties. 

We also note that the Thurston pullback map for $f$ is analytic (by similar reasoning as in [\cite{buff_cui_tan}, Proposition 1.4]). Thus, it is contracting with respect to the \textit{Teichm\"uller metric}, given by 
\[
d([\varphi_1], [\varphi_2]) = \inf_{\psi} \log K(\psi)
\]
where $[\varphi_1], [\varphi_2] \in \mathcal{T}_f$, the infimum is taken over all quasiconformal homeomorphisms $\psi$ such that $\varphi_2 = \psi \circ \varphi_1$ on $P_f$ and $\varphi_2$ and $\psi \circ \varphi_1$ are homotopic rel $P_f$, and $K(\psi)$ is the quasiconformal constant of $\psi$. See [\cite{hubbardvol1}, Proposition and Definition 6.4.4] for more details. Note that this metric makes $\mathcal{T}_f$ into a complete metric space. 

\begin{remark}
    We will denote the Thurston pullback map by both $\sigma_f$ and by $\sigma$. 
\end{remark}

\begin{thm}[Relationship between cosine and fixed point of $\sigma_f$]\label{fpiffthursequiv}
    A postsingularly finite topological cosine map $f$ is combinatorially equivalent to a postsingularly finite holomorphic cosine map $C_\lambda$ if and only if $\sigma_f$ has a fixed point.
\end{thm}

\begin{proof}
     Suppose a postsingularly finite cosine map $f$ is combinatorially equivalent to a postsingularly finite holomorphic cosine map $C_\lambda$. Then, there exists homeomorphisms $\varphi, \varphi'$ as in Definition \ref{topcosthurstonequiv} with commutative diagram as follows. 
    \begin{center}
    \begin{tikzcd}
    S^2 \setminus\{\infty \} \arrow[d, "f"'] \arrow[r, "\varphi'"] & \mathbb{C} \arrow[d, "C_\lambda"] \\
    S^2 \setminus\{\infty \} \arrow[r, "\varphi"]                  & \mathbb{C}              
    \end{tikzcd}
    \end{center}
    So $\sigma_f([\varphi]) = [\varphi']$ and as $\varphi$ and $\varphi'$ are isotopic rel $P_f$, it follows that $[\varphi] = [\varphi']$.

    Now suppose $\sigma_f$ has a fixed point $[\varphi] \in \mathcal{T}_f$. Then, $f$ is combinatorially equivalent to $C_{\varphi(x_1)}$. Indeed $\varphi(P_f) = P_{C_{\varphi(x_1)}}$, $\varphi(\infty) = \infty$, and the diagram
    \begin{center}
    \begin{tikzcd}
    S^2 \setminus\{\infty \} \arrow[d, "f"'] \arrow[r, "\varphi"] & \mathbb{C} \arrow[d, "C_{\varphi(x_1)}"] \\
    S^2 \setminus\{\infty \} \arrow[r, "\varphi"]                  & \mathbb{C}              
    \end{tikzcd}
    \end{center}
    commutes by definition of $\sigma_f$.
\end{proof}

\section{Teichm\"uller Space in Relation to Quadratic Differentials}
\label{teich_quad_spaces_section}

In this section, we provide background on why quadratic differentials come into play for the proof of Theorem \ref{main_thm}. For further background on Teichm\"uller theory, see \cite{ahlfors}, \cite{gardiner_lakic}, \cite{hubbardvol1}, \cite{hubbardvol2}, \cite{imayoshi_taniguchi}, and \cite{lehto}.

\subsection{Basics of quadratic differentials}

Quadratic differentials form the backbone of establishing Theorem \ref{main_thm}. We provide basic definitions needed for our results. See [\cite{hubbardvol1}, Chapters 5-6] for information beyond what is defined here. Suppose $X$ is a hyperbolic Riemann surface. 

\begin{defn}\label{quad_diff_def}
    A \textit{holomorphic quadratic differential} is a section of the tensor square of the sheaf of holomorphic 1-forms. The vector space of holomorphic quadratic differentials is denoted $Q(X)$.  
\end{defn}

For our purposes, it will be more useful to think of $q \in Q(X)$ locally: $q = q(z)dz^2$ for $(U, z)$ in an atlas. The following space will be a key player in our calculations. 

\begin{defn}\label{int_quad_diff_def}
    The space of \textit{integrable quadratic differentials} is 
    \[Q^1(X) := \{q \in Q(X) : \norm{q}_1 < \infty \} 
    \]
    where the norm is the $L^1$ norm $\norm{q}_1 = \int_X |q|$ and the norm of $q$ is called its \textit{mass}.
\end{defn}

If $X = \hat{\mathbb{C}} \setminus A$ where $A$ is a finite set such that $|A| > 3$, we will use the following notation: $Q^1(A) := Q^1(\hat{\mathbb{C}} \setminus A)$. Quadratic differentials in this space can have at most simple poles in $A$ due to the requirement of being integrable. The notation $\normc{\cdot}$ will be used rather than the usual $L^1$ norm notation.

\subsection{Cotangent space to Teichm\"uller space} Let $[\varphi] \in \mathcal{T}_f$. Rather than working with the tangent space, $T_{[\varphi]}\mathcal{T}_f$ of $\mathcal{T}_f$, the cotangent space will prove more helpful. The cotangent space to $\mathcal{T}_f$, denoted $\mathcal{T}^*_{[\varphi]} \mathcal{T}_f$, is isomorphic to $Q^1(\varphi(P_f))$. This space is preferable to work with because the coderivative of $\sigma$, $d\sigma^*_{[\varphi']}: \mathcal{T}^*_{[\varphi']} \mathcal{T}_f \rightarrow \mathcal{T}^*_{[\varphi]} \mathcal{T}_f$, is given by pushing forward $q \in Q^1(\varphi'(P_f))$. That is, for $q \in Q^1(\varphi'(P_f))$,
\[
d\sigma^*_{[\varphi']}q=(C_{\varphi(x_1)})_*q.
\]
This follows in a similar way to [\cite{hubbardvol2}, Proposition 10.7.2]. 

\begin{remark}
    Let $q \in Q^1(\varphi'(P_f))$. One might think that a simple pole of $q$ may be made worse (that is, the multiplicity goes up) by pushing forward by $C_{\varphi(x_1)}$. Remarkably, the simple poles of $q$ stay simple poles of $(C_{\varphi(x_1)})_*q$. It is possible that $(C_{\varphi(x_1)})_*q$ may obtain more poles at the critical values but these poles are again at most simple. Therefore, $(C_{\varphi(x_1)})_*q$ is indeed in $Q^1(\varphi(P_f))$.
\end{remark}

The norm of the coderivative is therefore
\[
\norm{d\sigma^*_{[\varphi']}} = \norm{(C_{\varphi(x_1)})_*} = \sup \left\{\frac{\normc{(C_{\varphi(x_1)})_*q}}{\normc{q}} : q \in Q^1(\varphi'(P_f)) \setminus \{0 \} \right\}.
\]
Because the tangent space, and thus the cotangent space, to Teichm\"uller space is finite-dimensional, the derivative $d\sigma_{[\varphi]}: T_{[\varphi]} \mathcal{T}_f \rightarrow T_{[\varphi']} \mathcal{T}_f$ satisfies $\norm{d\sigma_{[\varphi]}} = \norm{d\sigma^*_{[\varphi']}}$. As in the exponential case, these norms are strictly less than 1. This is further explained in Section \ref{cos_eff_quad_diff_section} or one can also refer to [\cite{expthursmaps}, Section 3.2].

\begin{defn}\label{symmetric_poles}
    If $q$ is an integrable meromorphic quadratic differential with at most finitely many poles on $\hat{\mathbb{C}}$, then a pole of $q$ is called \textit{cos-symmetric} if there exists $k \in \mathbb{Z}$ and $z_0 \in \mathbb{C}^*$ such that the pole can be written as $k\pi + z_0$ and there exists another pole of $q$ at $k\pi - z_0$.
\end{defn}

Example \ref{example_eff_qd_with_cp} provides a sequence of quadratic differentials with four cos-symmetric poles (where $k=0$). In the context of Theorem \ref{main_thm}, quadratic differentials with cos-symmetric poles provide a potential way for mass to concentrate at a critical point and the quadratic differential appears to push forward with little loss of mass.

\subsection{The mass condition} 

Establishing Theorem \ref{main_thm} involves understanding where the mass of integrable meromorphic quadratic differentials on $\hat{\mathbb{C}}$ concentrate. The quadratic differentials of interest are those that push forward under cosine with little loss of mass. We initially thought that such quadratic differentials could not have mass concentrating at a critical point, but Example \ref{example_eff_qd_with_cp} provides evidence to the contrary. As a result, we impose a condition that keeps track of where the critical points are relative to where the mass concentrates. This condition comes from the adaptation to cosine of the key proposition (Proposition 3.2) in \cite{expthursmaps}. Proposition \ref{key_prop} is this adaptation (in its statement, the modulus should be thought of as large and $r$ should be thought of as sufficiently close to 1). 

\begin{prop}[Cos-efficiency gives rise to annuli]\label{key_prop}
    For all poles $N$, for all iterates $m$, and for all nonzero moduli $M$, there exists $r \in (0,1)$ such that the following property holds: if $q$ is an integrable meromorphic quadratic differential with at most $N$ poles and for all $\lambda_1, \ldots, \lambda_m \in \mathbb{C}^*$ such that $\normc{(C_{\lambda_m} \circ \cdots \circ C_{\lambda_1})_*q} > r\normc{q}$, there exist concentric disks $\tilde{D} \subset D$ such that 
    \begin{enumerate}[(1)]
        \item $\tilde{D}$ contains at least two poles of $q$,
        \item $\modulus (D \setminus \tilde{D}) \geq M$, and 
        \item if $k\pi \notin C_{\lambda_s} \circ \cdots \circ C_{\lambda_1}(D)$ for all $k \in \mathbb{Z}$ and $0 \leq s \leq m$, then $C_{\lambda_m} \circ \cdots \circ C_{\lambda_1}$ is injective on $D$.
    \end{enumerate}
\end{prop}

 \begin{condition}\label{condition:mass_condition}
     In the setting of Proposition \ref{key_prop}, $k\pi \notin C_{\lambda_s} \circ \cdots \circ C_{\lambda_1}(D)$ for all $k \in \mathbb{Z}$ and $0 \leq s \leq m$.
\end{condition}

The \hyperref[condition:mass_condition]{mass condition} intuitively says that if a quadratic differential pushes forward with little loss of mass under several iterations of a cosine map, and if the mass concentrates in small disks, then the images of these disks stay away from critical points of cosine. The relevance of this condition is that $C_\lambda$ (where $\lambda \in \mathbb{C}^*)$ is not locally affine in a neigbourhood of a critical point and, thus, pushing forward a small neighbourhood is not like the exponential case.

Note also that the \hyperref[condition:mass_condition]{mass condition} is nondynamical. However, when used to understand postsingularly finite topological cosine maps in Proposition \ref{contraction_or_levy_cycle}, the condition is not actually needed for quadratic differentials with very specific configurations of their poles - this is discussed in Remark \ref{some_q_do_not_occur}.

\begin{remark}
    The \hyperref[condition:mass_condition]{mass condition} can be eliminated in Theorem \ref{main_thm} if it can be shown that if a quadratic differential has mass that concentrates in a neighbourhood of a critical point of cosine, then it must push forward with a definite loss of mass under some number of iterations $C_{\lambda_s} \circ \cdots \circ C_{\lambda_1}$ where $\lambda_1, \ldots, \lambda_s \in \mathbb{C}^*$.
\end{remark}

\subsubsection{Proofs of key results}

Here, we provide analogous statements for cosine from \cite{expthursmaps} (specifically Proposition 3.3 and Theorem 2.4). 

\begin{prop}[Existence of Levy cycle]\label{contraction_or_levy_cycle}
    Assume the \hyperref[condition:mass_condition]{mass condition} holds. Let $f$ be a postsingularly finite topological cosine map such that all critical points are strictly preperiodic. Let $k$ be the sum of the preperiod and period of the critical orbit. For every $d_0 > 0$, there exists $r \in (0,1)$ such that if $\tau \in \mathcal{T}_f$ with $d(\tau, \sigma(\tau)) < d_0$ and $\norm{d\sigma^{\circ k}(\tau)} > r$, then $f$ has a degenerate Levy cycle.
\end{prop}
 
\begin{proof}
    This proof is almost the same as [\cite{expthursmaps}, Proposition 3.3], and so we highlight the key differences. Note that there is one extra point in $P_f$ but it causes no problems due to the merging of the critical orbits and that we are in the preperiodic setting. Let $d_0 > 0$. For $|P_f| = k+1$ poles, $k$ iterates, and modulus $M:= \frac{\pi}{\log (3+2\sqrt{2})}ke^{kd_0}$, Proposition \ref{key_prop} gives us the existence of some $r \in (0,1)$; we show this $r$ works in this setting. 
    
    Let $\tau \in \mathcal{T}_f$ with $d(\tau, \sigma(\tau)) < d_0$ and $\norm{d\sigma^{\circ k}(\tau)} > r$. If $\tau = [\varphi_0]$ and $[\varphi_{n}] := \sigma^{\circ n}([\varphi_0])$ for $n \geq 1$, then note that
    \begin{align*}
        \norm{(d\sigma^{\circ k})^*_{[\varphi_n]}} &= \norm{(C_{\lambda_k} \circ \cdots \circ C_{\lambda_1})_*} \\
        &= \sup \left\{\frac{\normc{(C_{\lambda_k} \circ \cdots \circ C_{\lambda_1})_*q}}{\normc{q}} : q \in Q^1(\varphi_{n+k}(P_f)) \setminus \{0 \} \right\}
    \end{align*}
    where each $\lambda_i = \varphi_{n+k-i}(x_1)$ for $1 \leq i \leq k$. Thus, the assumption $\norm{d\sigma^{\circ k}(\tau)} > r$ implies that there exists $q \in Q^1(\varphi_k(P_f))$ such that
    \[
    \normc{(C_{\lambda_k} \circ \cdots \circ C_{\lambda_1})_*q} > r\normc{q}.
    \]
    Proposition \ref{key_prop} then implies that there exists concentric disks $\tilde{D} \subset D$ such that $\tilde{D}$ contains at least two poles of $q$ and $\modulus(D \setminus \tilde{D}) \geq M$. The main difference between the exponential and cosine case occurs here - understanding where $D$ is relative to the critical points of cosine. Because the \hyperref[condition:mass_condition]{mass condition} holds, $k\pi \notin C_{\lambda_s} \circ \cdots \circ C_{\lambda_1}(D)$ for all $k \in \mathbb{Z}$ and $0 \leq s \leq k$. Thus, by (3) of Proposition \ref{key_prop}, $C_{\lambda_s} \circ \cdots \circ C_{\lambda_1}$ is injective on $D$ for $0 \leq s \leq k$.
    
    At this point, the proof continues as in the exponential case - it is possible to extract a degenerate Levy cycle from $k$ annuli, which are essential in the complement of the postsingular set due to the fact that both critical values can never be contained in the same component such an annulus bounds.
\end{proof}

\begin{remark}\label{some_q_do_not_occur}
    Imposing the mass condition comes from candidate cos-efficient quadratic differentials like Example \ref{example_eff_qd_with_cp} in the next section. But, in fact, such quadratic differentials may not arise in the setting of Proposition \ref{contraction_or_levy_cycle} so it is possible to make the \hyperref[condition:mass_condition]{mass condition} less restrictive with specific quadratic differentials. We outline some cases where the \hyperref[condition:mass_condition]{mass condition} need not be used in the proof of Proposition \ref{contraction_or_levy_cycle}. These cases have cos-symmetric poles concentrated near a critical point. 

    Let $f$ be a postsingularly finite topological cosine map such that all critical points are strictly preperiodic with $k$ denoting the sum of the preperiod and period of the critical orbit. Let $d_0 > 0$. Take $k+1$ poles, $k$ iterates, and modulus $M$ as in the proof of Proposition \ref{contraction_or_levy_cycle}. Use these to obtain an $r \in (0,1)$ from Proposition \ref{key_prop}. Let $\tau \in \mathcal{T}_f$ with $d(\tau, \sigma(\tau)) < d_0$ and $\norm{d\sigma^{\circ k}(\tau)} > r$. For $\tau = [\varphi_0]$, let $[\varphi_{n}] := \sigma^{\circ n}([\varphi_0])$ for $n \geq 1$. Translate the assumption $\norm{d\sigma^{\circ k}(\tau)} > r$ to the existence of $q \in Q^1(\varphi_k(P_f))$ such that
    \[
    \normc{(C_{\lambda_k} \circ \cdots \circ C_{\lambda_1})_*q} > r\normc{q}.
    \]
    where each $\lambda_i = \varphi_{k-i}(x_1)$ for $1 \leq i \leq k$. There exists concentric disks $\tilde{D} \subset D$ such that $\tilde{D}$ contains at least two poles of $q$ and $\modulus(D \setminus \tilde{D}) \geq M$. Suppose there are more than two poles of $q$ in $\tilde{D}$, all of which are cos-symmetric (see Example \ref{example_eff_qd_with_cp}). We outline some cases below (that do not require the \hyperref[condition:mass_condition]{mass condition}) according to whether the critical values in $\varphi_k(P_f)$ are in $\tilde{D}$.
    \begin{itemize}
        \item If $\varphi_k(x_1)$ and $\varphi_k(y_1)$ are not in $\tilde{D}$, then it follows that $f$ maps at least four points of $P_f$ that are not critical values to two points of $P_f$. A postsingularly finite topological cosine map cannot have this occur. 

        \item If exactly one of $\varphi_k(x_1)$ and $\varphi_k(y_1)$ is in $\tilde{D}$, then $f$ maps at least three points of $P_f$ that are not critical values to two points of $P_f$. Once again, a postsingularly finite topological cosine map cannot have this occur.
        
        \item Suppose $\varphi_k(P_f) = \varphi_{k-1}(P_f)$ and $\varphi_k(x_1), \varphi_k(y_1) \in \tilde{D}$. The image $C_{\lambda_1}(\tilde{D})$ is a topological disk that contains at least 3 poles in $\varphi_k(P_f)$, of which one is $\lambda_1$. Because we have $\normc{(C_{\lambda_1})_*q} > r\normc{q}$, $C_{\lambda_1}(\tilde{D})$ intersects $\tilde{D}$ in such a way that implies that $\tilde{D}$ contains at least one more pole of $q$ than it already does. Thus, this case cannot occur. In fact, this particular case allows for $q$ to have exactly two poles of $q$ in $\tilde{D}$ with the same contradiction occurring.  
    \end{itemize} 
    Therefore, such quadratic differentials $q$ do not arise as a result of the Thurston iteration. In general, the combinatorial contradiction that occurs above happens when $\tilde{D}$ contains at least six cos-symmetric poles of $q$ (i.e., $f$ maps at least four points of $P_f$ that are not critical values to two points of $P_f$, something that cannot occur for a postsingularly finite topological cosine map). Thus, if these specific quadratic differentials arise in the nondynamical setting that will be discussed in Section \ref{cos_eff_quad_diff_section}, they cannot arise in the dynamical setting of Proposition \ref{contraction_or_levy_cycle} and hence Theorem \ref{main_thm}. 
\end{remark}

\begin{proof}[Proof of Theorem \ref{main_thm}]
    This proof is the same as [\cite{expthursmaps}, Main Theorem 2.4]. The general idea is to run the Thurston iteration on  $\sigma_f: \mathcal{T}_f \rightarrow \mathcal{T}_f$. By identifying the coderivative with the cosine push-forward, one can then appeal to Proposition \ref{contraction_or_levy_cycle} to obtain a degenerate Levy cycle.
\end{proof}

\section{Cosine Efficient Quadratic Differentials}
\label{cos_eff_quad_diff_section}

\subsection{Pushing forward by cosine.}\label{pushing_forward_subsection}

Given a covering map, the push-forward of an integrable quadratic differential can be calculated as the direct image operator (see [\cite{hubbardvol1}, Definition 5.4.15]). Let $[\varphi] \in \mathcal{T}_f$, $q \in Q^1(\varphi(P_f))$, and write $q = q(z)dz^2$. For the simplest case of $\lambda = 1$, if $w = \cos(z)$, then $dw = -\sin(z)dz$ and thus we have
\[
\cos_*(q(z)dz^2) = \sum_{z \in \cos^{-1}(w)} \frac{q(z)}{\sin^2(z)}dw^2 = \frac{dw^2}{1-w^2} \sum_{k \in \mathbb{Z}} q(2\pi k \pm \arccos(w)).
\]
A similar formula can be found for $C_\lambda(z) = \lambda\cos z$ but for the results of Section \ref{adaptations_to_cosine_section}, it suffices to understand the simpler case. As in the case of the exponential, this gives a concrete formula for what the coderivative of $\sigma$ can really look like. Understanding this push-forward is the heart of Section \ref{cosine_efficiency_subsection}.

\subsection{Cosine efficiency}
\label{cosine_efficiency_subsection}
The quality of being cosine efficient is a quality of a sequence of measurable quadratic differentials $(q_n)$. It helps us understand how the ratio of the mass of $\cos_*q_n$ compares to that of the original quadratic differential $q_n$. In general, the push-forward does not increase the $L^1$ norm so being cosine efficient means that for very large $n$, cosine pushes forward $q_n$ with almost negligible loss of mass. 

\begin{defn}
    Let $(q_n)$ be a sequence of measurable quadratic differentials. The sequence $q_n$ is called \textit{cos-efficient} if $\norm{q_n}_{\mathbb{C}} \neq 0$ and
    \[
        \lim_{n \to \infty} \frac{\norm{\cos_*q_n}_\mathbb{C}}{\norm{q_n}_\mathbb{C}} = 1. 
    \]
\end{defn}

\begin{defn}
    A non-zero meromorphic integrable quadratic differential $q$ is said to have \textit{absolutely efficient} push-forward by cosine if $\normc{\cos_*q} = \normc{q}$. 
\end{defn}

\begin{lemma}\label{cos-pushforward-not-abs-eff}
Let $q$ be a non-zero meromorphic integrable quadratic differential. Then the push-forward by $f(z) = \cos(az+b)$, where $a, b \in \mathbb{C}$ with $a \neq 0$, of $q$ cannot be absolutely efficient. 
\end{lemma}

\begin{proof}
    Suppose (for contradiction) that $q$ is absolutely efficient so that $\normc{f_*q} = \normc{q}$. Then, $q\left(z+\frac{2\pi}{a}\right) = c \cdot q(z)$ for some positive real constant $c$. But then $\normc{q} = \infty$, a contradiction.
\end{proof}

\subsection{Limit models.} 

Hubbard, Schleicher, and Shishikura provide a ``thick-thin" decomposition of quadratic differentials on a general hyperbolic Riemann surface of finite type [\cite{expthursmaps}, Theorem 4.1 (Decomposition of quadratic differentials)]. Given a sequence of measurable quadratic differentials on $\mathbb{C}$, this allows them to understand the limits of each of the pieces in their decomposition as one of two models (up to a subsequence) as well as how the mass of these quadratic differentials is decomposed. The definition that is key to this discussion is the following. 

\begin{defn}[\cite{expthursmaps}, Definition 5.1]\label{limit_model_def}
    Let $(q_n)$ be a sequence of measurable quadratic differentials on $\mathbb{C}$ with $0 < \normc{q_n} < \infty$ for each $n$, such that $\lim_{n \to \infty} \normc{q_n}$ exists and is non-zero.
    \begin{enumerate}[(1)]
        \item \textit{Thick case:} If $q$ is a meromorphic quadratic differential with $0 < \normc{q} < \infty$, then \textit{$(q_n)$ has limit model $q$} if there exists affine maps $M_n(z) = a_nz+b_n$ (with $a_n, b_n \in \mathbb{C}$) satisfying 
        \[
        \lim_{n \to \infty} \normc{(M_n)^*q_n - q} = 0.
        \]
        We call $M_n$ the \textit{scaling} and $a_n$ the \textit{scaling factor}.

        \item \textit{Thin case:} We say \textit{$(q_n)$ has limit model $\frac{dz^2}{z^2}$ on annuli}
        \[
        A_n = \{z \in \mathbb{C} : r_n < |z-b_n| < R_n \},
        \]
        where $0<r_n<R_n$ and $b_n \in \mathbb{C}$, if there exists $c_n \in \mathbb{C}^*$ such that $\frac{R_n}{r_n} \to \infty$ and the affine maps $M_n(z) = z + b_n$ satisfy
        \[
        \lim_{n \to \infty} \normc{(M_n)^*q_n - c_n\frac{dz^2}{z^2}\Big|_{\{r_n < |z| <R_n\}} } = 0.
        \]
        We call $r_n$ the \textit{inner radius of $A_n$}, $R_n$ the \textit{outer radius of $A_n$}, and $b_n$ the \textit{center of $A_n$}.
    \end{enumerate}
\end{defn}

\begin{remark}\label{b_n_is_a_pole}
In the thick case, it suffices to suffices to assume $q$ has a pole at $0$ because 
\[
\normc{(M_n)^*q_n - q} = \normc{(M_n \circ M)^*q_n - M^*q}
\]
for an affine transformation $M: \mathbb{C} \rightarrow \mathbb{C}$. Therefore, $M(0)=b_n$ corresponds to a pole of $q_n$. Similarly, for the thin case, we may assume $b_n$ corresponds to a pole of $q_n$.
\end{remark}

\begin{remark}\label{each_piece_cos_eff_remark}
    Suppose $(q_n)$ is a sequence of measurable quadratic differentials on $\mathbb{C}$. There exists a subsequence of $(q_n)$ such that in the decomposition of each $q_n$ into integrable quadratic differentials, one can study each of the limit models above separately [\cite{expthursmaps}, Theorem 5.2 (Decomposition of mass)]. The importance of this in relation to efficiency is, for instance, if one is trying to show that this sequence is not cos-efficient, then it suffices to work piece by piece and show that one piece is not cos-efficient. This will imply that as a whole, $(q_n)$ is not cos-efficient. This technique is employed in Section \ref{adaptations_to_cosine_section}. 
\end{remark}

We next present two examples of candidate cos-efficient sequences of quadratic differentials. The second example indicates why the \hyperref[condition:mass_condition]{mass condition} is present in Theorem \ref{main_thm}. Our examples aim to show that very little mass is lost when pushing forward by cosine and, furthermore, mass is indeed concentrated in very small disks.

\begin{figure}[b]
  \centering
  \begin{tabular}{@{}c@{}}
    \begin{tikzpicture}[ scale = 2 ]
        \draw[gray!30, fill=gray!30]    (1.25,0.1) -- (3.75,0.1) -- (3.75,-0.1) -- (1.25,-0.1) ;
        
        \draw[thick] (-1,1) node[left,black] {$1$}  -- (-1,-1) ;
        \draw[thick] (-1,1)  -- (1,1) ;
        \draw[thick] (-1,-1)  -- (1,-1) ;
        \draw[thick] (1,1)  -- (1,0.1) ;
        \draw[thick] (1,-1)  -- (1,-0.1) ;
        \draw[thick] (1,0.1)  -- (4,0.1) node[right,black] {$0$} ;
        \draw[thick] (1,-0.1)  -- (4,-0.1) node[right,black] {$\infty$};
        \draw[thick] (4,0.1)  -- (4,-0.1) ;
        
        \draw[|-|] (-1,-1.1)  -- (1,-1.1) node[midway, below] {$1$};
        \draw[|-|] (-1.25,1)  -- (-1.25,-1) node[midway, left] {$1$};
        \draw[|-|] (0.9,0.1)  -- (0.9,-0.1) node[midway, left] {$\frac{1}{n}$};
        \draw[|-|] (1.03,-0.2)  -- (4,-0.2) node[midway, below] {$n$};
        \end{tikzpicture} \\[\abovecaptionskip]
    \small (a) Represented as a polygon
  \end{tabular}

  \vspace{\floatsep}

  \begin{tabular}{@{}c@{}}
    \begin{tikzpicture}[ scale = 2 ]
        \draw[thick] (-2,0) -- (0,0) node[below]{$0$};
        \draw[thick, gray!30, fill=gray!30] (1,0) circle[radius=0.6];
        \draw[thick, gray!30, fill=white] (1,0) circle[radius=0.4];
    
        \draw (0,0) node{$\bullet$};
        \draw (1,0) node{$\bullet$} node[below]{$1$};
        \draw (0.9,0) node{$\bullet$};
        \draw (0.8,0) node{$\bullet$};
        \draw (1.1,0) node{$\bullet$};
        \draw (1.2,0) node{$\bullet$};
        \draw (1.3,0) node{$\bullet$};
        \end{tikzpicture} \\[\abovecaptionskip]
    \small (b) Represented in $\mathbb{C}$
  \end{tabular}

  \caption{A candidate cos-efficient sequence of quadratic differentials is pictured here. Each quadratic differential has 6 poles and 2 zeroes, and mass $4$ with the majority of its mass concentrated in a small neighbourhood of $1$. The polygonal representation shows that an annulus of very large modulus separates $1$ from $0$ and $\infty$, which then gives an interpretation for how close poles are in $\mathbb{C}$. The line connecting $0$ to $\infty$ in the polygonal representation becomes the negative real axis in $\mathbb{C}$, as shown.}\label{fig:quad_diff}
\end{figure}

\begin{example}\label{example_eff_qd}
    Figure \ref{fig:quad_diff} provides an example of what a sequence of cos-efficient quadratic differentials may be thought of as, both represented as a polygon and pictured in $\mathbb{C}$. In this illustration, note that the following logic works just as well for the exponential map (so in particular, the sequence of quadratic differentials can be thought of as exp-efficient). We will work with round annuli because an annulus of large modulus contains a round subannulus with the same modulus minus a small constant [\cite{mcmullen_complex_dynamics}, Theorem 2.1].
    
    Although this polygonal representation is a special case, its benefit is that we can assume that when gluing two copies along their common boundary, this common boundary is homeomorphic to $\mathbb{R} \cup \{\infty \}$. This is how we are able to obtain the representation of this quadratic differential in $\mathbb{C}$. For further background on the construction of quadratic differentials in this way, see \cite{hubbardvol1} and \cite{expthursmaps}. The shaded region will help us in determining how close these zeroes and poles are to each other. 

    \begin{figure}
    \centering
    \begin{tikzpicture}[ scale = 2.2 ]
        \draw[thick, gray!30, fill=gray!30] (1,0) circle[radius=0.7];
        \draw[thick, gray!30, fill=white] (1,0) circle[radius=0.5];
        \draw[thick] (1,0) circle[radius=0.8];
        \draw[thick] (1,0) circle[radius=0.3];
    
        \draw (0.2,0) node{$\bullet$} node[left]{$-1$};
        \draw (1,0) node{$\bullet$} node[below]{$0$};
        \draw (0.85,0) node{$\bullet$};
        \draw (0.7,0) node{$\bullet$};
        \draw (1.1,0) node{$\bullet$};
        \draw (1.2,0) node{$\bullet$};
        \draw (1.3,0) node{$\bullet$};

        \draw (1.25,0) node[right]{$R_1$};
        \draw (1.433013,0.25) node[]{$R_2$};
        \draw (1.35,0.606218) node[]{$R_3$};
    \end{tikzpicture}
    \caption{Conformal mapping of Figure \ref{fig:quad_diff} with various radii labeled}
    \label{fig:quad_diff_conf_map}
\end{figure}
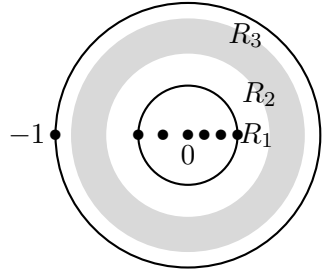
    
    To see why this sequence of quadratic differentials is a candidate cos-efficient sequence, conformally map Figure \ref{fig:quad_diff} by $z \mapsto \sqrt{z}$ and then $z \mapsto \frac{z-1}{z+1}$ to obtain Figure \ref{fig:quad_diff_conf_map}. Let 
    \[
    A(r, R) := \{z \in \mathbb{C} : r < |z| < R \}.
    \]
    Then, think of how the exponential map maps the cylinder of circumference $\frac{1}{n}$ and height $n$ onto $A(R_1, 1)$. By taking the padding (or unshaded) regions of this cylinder to be $\frac{1}{n}$, it follows that the modulus of $A(R_1, R_2)$ and $A(R_3, 1)$ is $n$ and the modulus of $A(R_2, R_3)$ is $n(n-2)$. Therefore, $R_1 = e^{-2\pi n^2}$, $R_2 = {e^{-2\pi n(n-1)}}$, and $R_3 = e^{-2\pi n}$. Although we must conformally map back to the original plane, the point is that Figure \ref{fig:quad_diff_conf_map} is wildly out of proportion in showing how close the zeroes and poles in the disk of radius $R_3$ are to each other (for instance, think of this picture when $n=10$). Furthermore, this shows that the quadratic differential in Figure \ref{fig:quad_diff} has the majority of its mass concentrated in a very small neighbourhood of $1$. Note that thinking from the point of view of either cosine or the exponential, this neighbourhood maps forward by a locally affine map and thus there is almost neglible mass lost in the push-forward. 
\end{example}

\begin{figure}
  \centering
  \begin{tabular}{@{}c@{}}
    \begin{tikzpicture}[ scale = 2 ]
    \draw[gray!30, fill=gray!30]    (1.25,0.1) -- (3.75,0.1) -- (3.75,-0.1) -- (1.25,-0.1) ;
    
    \draw (-1,1) node[left,black] {$-a_n$}  -- (-1,-1) node[midway, left] {$0$};
    \draw (-1,1)  -- (1,1) ;
    \draw (-1,-1) node[left,black] {$a_n$}  -- (1,-1) ;
    \draw (1,1) node[right,black] {$-b_n$}  -- (1,0.1) ;
    \draw (1,-1) node[right,black] {$b_n$}  -- (1,-0.1) ;
    \draw (1,0.1)  -- (4,0.1) node[right,black] {$\infty$} ;
    \draw (1,-0.1)  -- (4,-0.1) node[right,black] {$1$};
    \draw (4,0.1)  -- (4,-0.1) ;
    
    \draw[|-|] (-1,-1.1)  -- (1,-1.1) node[midway, below] {$1$};
    \draw[|-|] (-1.5,1)  -- (-1.5,-1) node[midway, left] {$1$};
    \draw[|-|] (0.9,0.1)  -- (0.9,-0.1) node[midway, left] {$\frac{1}{n}$};
    \draw[|-|] (1.03,-0.2)  -- (4,-0.2) node[midway, below] {$n$};
    \end{tikzpicture} \\[\abovecaptionskip]
    \small (a) Represented as a polygon
  \end{tabular}

  \vspace{\floatsep}

  \begin{tabular}{@{}c@{}}
    \begin{tikzpicture}[ scale = 2 ]
    \draw[thick] (1,0) node[below]{$1$} -- (3,0);
    \draw[thick, gray!30, fill=gray!30] (-1,0) circle[radius=0.7];
    \draw[thick, gray!30, fill=white] (-1,0) circle[radius=0.5];

    \draw (1,0) node{$\bullet$};
    \draw (-0.9,0) node{$\bullet$}; \draw (-0.8,-0.11) node{$a_n$};
    \draw (-0.7,0) node{$\bullet$}; \draw (-0.6,-0.11) node{$b_n$};
    \draw (-1.1,0) node{$\bullet$}; \draw (-1.13, -0.11) node{$-a_n$};
    \draw (-1.3,0) node{$\bullet$}; \draw (-1.38, -0.11) node{$-b_n$};
    \end{tikzpicture} \\[\abovecaptionskip]
    \small (b) Represented in $\mathbb{C}$
  \end{tabular}

  \caption{A candidate cos-efficient sequence of quadratic differentials is pictured here. Each quadratic differential has 6 poles and 2 zeroes, and mass $4$ with the majority of its mass concentrated in a small neighbourhood of $0$. The interpretation of the polygonal and plane representation is similar to Figure \ref{fig:quad_diff}.}\label{fig:quad_diff_with_cp}
\end{figure}

The sequence of quadratic differentials in Example \ref{example_eff_qd} is an example that can be encountered in \cite{expthursmaps} because the exponential map locally behaves like an affine map. In contrast, cosine locally behaves as either a degree two map or an affine map, according to whether a critical point is or is not present. Pushing forward a quadratic differential in general involves a sum over preimages of a point in the codomain (see Section \ref{pushing_forward_subsection} for the case of cosine). As a result, when constructing quadratic differentials that push forward with very little loss of mass, it is more likely for cancellation to occur in this sum if there is more than one term (i.e., the map is not locally affine). Therefore, our initial intuition led us to think that if a sequence of quadratic differentials is cos-efficient, then its mass does not concentrate near a critical point. However, a unique property of cosine is that it is an even function and so it is possible to build sequences of quadratic differentials with cos-symmetric poles that appear to push forward with a negligible loss of mass near a critical point. The next example is one such construction. 

\begin{example}\label{example_eff_qd_with_cp}
    Figure \ref{fig:quad_diff_with_cp} shows a sequence of quadratic differentials, represented as a polygon and also in $\mathbb{C}$. Each quadratic differential in this sequence has cos-symmetric poles in the bounded component of the annulus that separates $0$ from $1$ and $\infty$. As in Example \ref{example_eff_qd}, as $n \to \infty$, the annulus separating $1$ and $\infty$ from the other poles, grows in modulus and squeezes down the majority of the mass of the quadratic differential into a very small disk containing $-b_n$, $-a_n$, $a_n$, and $b_n$. Thus, $a_n$ and $b_n$ are sufficiently close to $0$ and once again, the scale of the diagram is out of proportion in how close poles are to each other. The critical point $0$ is not a pole, but where it is relative to the poles is important because most of the mass is concentrated in a very small disk with $0$ in the center of the poles $-b_n$, $-a_n$, $a_n$, and $b_n$. One may think of each quadratic differential in this sequence as 
    \[
    \frac{p(z)}{(z^2-a_n^2)(z^2-b_n^2)(z-1)}dz^2,
    \]
    where $p(z)$ is a quadratic polynomial so that $\infty$ is indeed a simple pole. Observe that in the small disk where the majority of the mass is concentrated, cosine is locally a degree two map. The sum that appears in the definition of the push-forward by cosine does not have terms cancelling completely because cosine is an even function. For example, if we compute this push-forward using just the first two terms of the power series centered at $0$ for cosine, then assuming $g(z) := 1-\frac{z^2}{2}$ and $w = g(z)$, we obtain
    \[
    g_*q = \frac{dw^2}{2-2w}\left(q(\sqrt{2-2w}) + q(-\sqrt{2-2w}) \right).
    \]
    If $q$ has a formula like the quadratic differential above, the terms in the sum are close to adding perfectly due to the symmetry of the poles of $q$ near $0$. 
    Therefore, this is a candidate for a sequence of quadratic differentials that is cos-efficient. 
\end{example}

\begin{remark}
    If Example \ref{example_eff_qd_with_cp} is rigorously proved to be cos-efficient, then it may be possible that perturbing the poles keeps the quadratic differential cos-efficient. That is, we replace $-b_n, -a_n, a_n$, and $b_n$ with $-b_n+\varepsilon_{1,n}, -a_n+\varepsilon_{2,n}, a_n+\varepsilon_{3,n}$, and $b_n+\varepsilon_{4,n}$, respectively, for $\varepsilon_{1,n}, \varepsilon_{2,n}, \varepsilon_{3,n}$, $\varepsilon_{4,n}$ with sufficiently small modulus. The possibility of sequences of such quadratic differentials being cos-efficient is precisely why the \hyperref[condition:mass_condition]{mass condition} accounts for mass concentrated near critical points with poles that are not necessarily cos-symmetric.
\end{remark}

\section{Adaptations of Exponential Results to Cosine}
\label{adaptations_to_cosine_section}

In this section, we produce similar results as in [\cite{expthursmaps}, Section 6] modified for cosine, specifically Proposition 3.2, Proposition 6.1, and Proposition 6.2. Much of these results have a similar analysis to the exponential case, with the differences being in accounting for the existence of critical points. These results have to do with sequences of quadratic differentials so they will split into the two limit model cases as described in Definition \ref{limit_model_def}. We refer the reader to Figure \ref{change_of_coords_fig} throughout this section as a way of recalling what plane the different objects are thought of in.

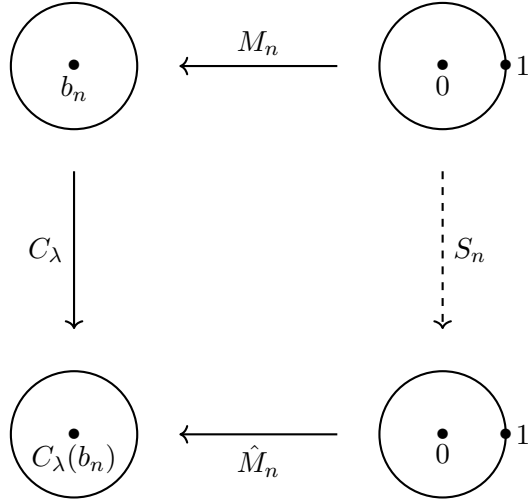
\begin{figure}
    \centering
    \begin{tikzpicture}[ scale = 1.4 ]
        \draw[thick] (2,2) circle[radius=0.6];
        \draw[thick] (-1.5,2) circle[radius=0.6];
        \draw[thick] (-1.5,-1.5) circle[radius=0.6];
        \draw[thick] (2,-1.5) circle[radius=0.6];
    
        \draw[thick, ->] (1,2) -- (-0.5,2) node[midway, above]{$M_n$};
        \draw[thick, ->] (-1.5,1) -- (-1.5,-0.5) node[midway, left]{$C_\lambda$};
        \draw[thick, ->] (1,-1.5) -- (-0.5, -1.5) node[midway, below]{$\hat{M}_n$};
        \draw[thick, dashed, ->] (2,1) -- (2,-0.5) node[midway, right]{$S_n$};
    
        \draw (2.6,2) node{$\bullet$} node[right]{$1$};
        \draw (2,2) node{$\bullet$} node[below]{$0$};
        \draw (-1.5,2) node{$\bullet$} node[below]{$b_n$};
        \draw (-1.5,-1.5) node{$\bullet$} node[below]{$C_\lambda(b_n)$};
        \draw (2.6,-1.5) node{$\bullet$} node[right]{$1$};
        \draw (2,-1.5) node{$\bullet$} node[below]{$0$};
    \end{tikzpicture}
    \caption{The scalings $\hat{M}_n$ in Proposition \ref{cos-eff-int-pushforwards} and Proposition \ref{cos_eff_annular_pushfowards} are chosen with the intent of $S_n \rightarrow \text{id}$ as $n \to \infty$ uniformly on compact subsets of $\mathbb{C}$. To do so in the context of cos-efficient quadratic differentials, if $w = M_n(z) = a_nz+b_n$, $M_n$ sends the unit disk in $\mathbb{C}$ to a small disk centered at $b_n$ in the $w$-plane. The importance of this is that if $b_n$ is not a critical point of $C_\lambda$, then  $C_\lambda$ acts like an affine map and sends this small disk to a small disk. One then forces the scaling $\hat{M}_n$ to be chosen so that $S_n(D_1(0))$ looks generally like $D_1(0)$ for large $n$.}
    \label{change_of_coords_fig}
\end{figure}

\subsection{Thick case}

\begin{prop}[Cos-efficient integrable push-forwards]\label{cos-eff-int-pushforwards}
    Let $(q_n)$ be a sequence of measurable quadratic differentials on $\mathbb{C}$. Suppose $(q_n)$ has limit model $q$, where $q$ is a meromorphic quadratic differential with $0 < \norm{q}_\mathbb{C} < \infty$ and scalings  $M_n(z) = a_n z + b_n$. If $(q_n)$ is cos-efficient, then 
    \begin{enumerate}[(1)]
        \item $a_n \to 0$ as $n \to \infty$ and 
        \item if the poles of $(q_n)$ do not converge to critical points of cosine as $n \to \infty$, then $(\cos_*q_n)$ has limit model $q$ with scalings $\hat{M}_n(z) = -a_n\sin(b_n)z + \cos(b_n)$. 
    \end{enumerate}
\end{prop}

\begin{proof}
    The proof is similar to [\cite{expthursmaps}, Proposition 6.1]. Beginning with (1), the proof is done by contradiction according to whether $a_n$ is bounded away from 0 and $\infty$ or $a_n \to \infty$. If either of these cases are assumed, $(q_n)$ cannot be cos-efficient which will be the required contradiction. By the Decomposition of Quadratic Differentials [\cite{expthursmaps}, Theorem 4.1], for every $n$ we can write
    \[
    q_n = \sum_{j} p_n^j,
    \]
    where each $p_n^j$ is a global integrable meromorphic quadratic differential and by the Decomposition of Mass Theorem [\cite{expthursmaps}, Theorem 5.2], it suffices to consider the $p_n^j$ with a definite share of the mass (see Remark \ref{each_piece_cos_eff_remark}). So from here onward we consider such $p_n^j$ in our estimates. 

    Because the limit model of $p_n^j$ is $q$, $\lim_{n \to \infty} \normc{(M_n)^*p_n^j - q} = 0$. So given $\varepsilon > 0$, there exists $N > 0$ such that if $n > N$, then $\normc{(M_n)^*p_n^j - q} < \varepsilon \normc{q}$. Then, 
    \[
    \normc{q} - \normc{(M_n)^*p_n^j} \leq \normc{(M_n)^*p_n^j-q} < \varepsilon \normc{q}
    \]
    so that
    \[
    \normc{q} < \frac{\normc{(M_n)^*p_n^j}}{1-\varepsilon}.
    \]
    Unlike the exponential case, translations may affect cos-efficiency. However, it suffices to modify $(q_n)$ so that the $b_n$ are bounded. Then, there exists a subsequence of the $b_n$ that converges to $b \in \mathbb{C}$. Pass to this subsequence and continue the notation $b_n$ for ease of notation.

    Suppose $a_n$ is bounded in $\mathbb{C}^*$. The special case of trivial scalings is treated in [\cite{expthursmaps}, Proposition 6.1] and its proof is identical for cosine. For the nontrivial scalings case, boundedness implies the existence of a convergent subsequence of $a_n$ that converges to $a \in \mathbb{C}^*$. We keep the same notation for $a_n$ as we pass to this subsequence. Let $f_{n}(z) = \cos\left(a_nz + b_n\right)$ and $f(z) = \cos\left(az + b\right)$. Then, $f_n$ converges uniformly to $f$ on compact subsets of $\mathbb{C}$ so if $N$ is chosen sufficiently large and $n > N$, we have $\normc{(f_n)_*q - f_*q} < \varepsilon\normc{q}$.
    By Lemma \ref{cos-pushforward-not-abs-eff}, there exists $c_q \in (0,1)$ such that $\normc{f_*q} \leq c_q \normc{q}$. For $n > N$, 
    \begin{align*}
        \normc{\cos_*p_n^j} &= \normc{(f_n)_*(M_n)^*p_n^j} \\
        &= \normc{(f_n)_*((M_n)^*p_n^j - q + q)} \\
        & \leq \normc{(f_n)_*((M_n)^*p_n^j-q)} + \normc{(f_n)_*q - f_*q + f_*q} \\
        &\leq \normc{(M_n)^*p_n^j - q} + \normc{(f_n)_*q - f_*q} + \normc{f_*q} \\
        &< \varepsilon\normc{q} + \varepsilon\normc{q} + c_q\normc{q} \\
        &\leq \frac{2\varepsilon + c_q}{1-\varepsilon}\normc{(M_n)^*p_n^j} \\
        &= \frac{2\varepsilon + c_q}{1-\varepsilon}\normc{p_n^j}.
    \end{align*}
    Choosing $\varepsilon$ sufficiently small, it follows that $(p_n^j)$ is not cos-efficient hence neither is $(q_n)$ contradicting our initial assumption. This concludes the case of bounded $a_n$ in $\mathbb{C}^*$. 

    Now suppose $a_n \to \infty$. Since a subsequence of $b_n$ converges to $b \in \mathbb{C}$, it suffices to assume $b_n = b$ for $n$ sufficiently large. We can rewrite $a_n$ such that $a_n = 2^{k_n}a_n^\sharp$ where $k_n \in \mathbb{N}$. This is so that $1 \leq |a_n^\sharp| < 2$ for $n$ sufficiently large. So $a_n^\sharp$ is bounded in $\mathbb{C}^*$ for $n$ sufficiently large. Take $Q(z) = 2z^2-1$, and $M_n^\sharp(z) := a_n^\sharp z + \frac{b}{2^{k_n}}$. Then,
    \[
    Q \circ \cos \circ M_n^\sharp(z) = Q\left(\cos\left(a_n^\sharp z +\frac{b}{2^{k_n}}\right)\right) = \cos\left(2a_n^\sharp z + \frac{b}{2^{k_n-1}} \right)
    \]
    hence
    $Q^{k_n} \circ \cos \circ M_n^\sharp(z) = \cos \circ M_n(z)$. Let $r_n^j := (M_n)^*p_n^j$. Then,
    \[
    \normc{\cos_*p_n^j} = \normc{\cos_*(M_n)_*r_n^j} = \normc{(Q^{\circ k_n})_*\cos_*(M_n^\sharp)_*r_n^j} \leq \normc{\cos_*(M_n^\sharp)_*r_n^j}
    \]
    and so 
    \[
    \frac{\normc{\cos_*p_n^j}}{\normc{p_n^j}} \leq \frac{\normc{\cos_*(M_n^\sharp)_*r_n^j}}{\normc{(M_n^\sharp)_*r_n^j}}.
    \]
 
    If $((M_n^\sharp)_*r_n^j)$ is cos-efficient, we are in the bounded case of the proof in which case we can reach the same contradiction. If $((M_n^\sharp)_*r_n^j)$ is not cos-efficient, then neither is $(p_n^j)$ by the inequality above and hence neither is $(q_n)$, a contradiction. Therefore, we may conclude that as $n \to \infty$, $a_n \to 0$ so (1) is proved. 
    
    Now we prove (2). By our hypothesis and Remark \ref{b_n_is_a_pole}, it follows that $b_n$ does not converge to $k\pi$ for any $k \in \mathbb{Z}$. We must show that $\normc{(\hat{M_n})^*(\cos_*p_n^j) - q} \to 0$ as $n \to \infty$. Start by defining $S_n := (\hat{M_n})^{-1} \circ \cos \circ M_n$. Then, 
    \[
    S_n(z) = \frac{\cos(a_nz+b_n) - \cos(b_n)}{-a_n\sin(b_n)} = z \cdot \frac{\cos(b_n+a_nz) - \cos(b_n)}{-a_nz} \cdot \frac{1}{\sin(b_n)}
    \]
    so as $n \to \infty$, $a_n \to 0$ and $b_n \to b$, where $b \neq k\pi$ for any $k \in \mathbb{Z}$. Therefore, $S_n(z) \to z$ as $n \to \infty$ uniformly on compact subsets of $\mathbb{C}$. The remainder of the proof does not rely on any properties of cosine so is identical to the exponential proof; we include it for completeness. Let $\varepsilon > 0$. For $N$ sufficiently large, it then follows that for $n>N$ we have $\normc{(M_n)^*p_n^j - q} < \frac{\varepsilon}{2}$. Next, \cite{expthursmaps} uses the following statement: if $h(z)$ is a meromorphic function in $D_1(0)$ with at most one simple pole in $D_{1/2}(0)$ and no poles in the $D_1(0) \setminus D_{1/2}(0)$, then for every $\varepsilon > 0$ there exists $\delta > 0$ such that $\norm{h}_{D_\delta(0)} \leq \varepsilon\norm{h}_{D_1(0)}$. In this situation, we can utilize this statement for each pole of $q$ (along with $\infty$) to build a neighbourhood $W$ of the poles of $q$ and $\infty$ that is dependent on some $\delta$ such that $\norm{q}_W < \frac{\varepsilon}{6}$ and $\norm{(S_n)_*q}_W < \frac{\varepsilon}{6}$ for sufficiently large $n$. This is done by using a change of coordinate on each pole (along with $\infty$) $p_i$ of $q$ to transport it to 0 as well as the fact that the poles of $(S_n)_*q$ converge to the poles of $q$ (along with $\infty$). Lastly, for $n$ sufficiently large, as $n \to \infty$, $S_n(z) \to z$ uniformly on $\hat{\mathbb{C}} \setminus W$ hence $\norm{(S_n)_*q - q}_{\hat{\mathbb{C}} \setminus W} < \frac{\varepsilon}{6}$. Combining all of these estimates and assuming that $N$ is sufficiently large, we have for $n > N$, 
    \begin{align*}
        \normc{(\hat{M_n})^*(\cos_*p_n^j) - q} &= \normc{(S_n)_*(M_n)^*p_n^j - q} \\
        &\leq \normc{(S_n)_*((M_n)^*p_n^j - q)} + \normc{(S_n)_*q - q} \\
        &\leq \normc{(M_n)^*p_n^j - q} + \norm{(S_n)_*q}_W +\norm{q}_W + \norm{(S_n)_*q - q}_{\hat{\mathbb{C}} \setminus W} \\
        &< \frac{\varepsilon}{2} + \frac{\varepsilon}{6} + \frac{\varepsilon}{6} + \frac{\varepsilon}{6} \\
        &= \varepsilon. 
    \end{align*}
\end{proof}

\begin{remark}
    In the above proof, note that $S_n$ has critical points at $\frac{k\pi-b_n}{a_n}$ for every $k \in \mathbb{Z}$ and thus, the translations $b_n$ should not approach a critical point of cosine so that the critical points of $S_n$ tend to infinity as $n$ tends to infinity. This ensures that $S_n$ will converge to the identity on compact subsets of $\mathbb{C}$.
\end{remark}

\subsection{Thin case}

\begin{lemma}[Cos-efficient annular distribution of mass]\label{cos-efficient-annular-dist-mass}
    For the quadratic differential $q:= \frac{dz^2}{z^2}$ and the annulus $A(r, R):= \{z \in \mathbb{C} : r < |z| < R \}$ (where $r < R$), there exists $\widetilde{\alpha} \in (0,1)$ such that for all $r, R$ satisfying $r \geq \pi$ and $R \geq 3r$, $\normc{\cos_*(q|_{A(r,R)})} \leq \widetilde{\alpha}\normc{q|_{A(r,R)}}$. 
\end{lemma}

\begin{proof}
   The proof is an adaptation of the ``toy example" of [\cite{expthursmaps}, Proposition 6.2], adjusted for cosine. Considering the points $1+i\pi, 1+2\pi + i\pi \in A(\pi, 3\pi)$, it follows that on this annulus there is cancellation of mass. This implies that there exists $\alpha < 1$ such that we have the inequality $\normc{\cos_*(q|_{A(\pi,3\pi)})} \leq \alpha\normc{q|_{A(\pi,3\pi)}}$. Next, let $F_i(z) = 3^iz$ (where $i \in \mathbb{N}$) and $Q(z) = 4z^3-3z$. For $r \geq \pi$, $\cos \circ F_i(z) = Q^{\circ i} \circ \cos(z)$ and so
   \[
        \normc{\cos_*(q|_{A(3^ir, 3^{i+1}r)})} = \normc{\cos_*(F_i)_*(q|_{A(r,3r)})} 
        = \normc{(Q^{\circ i})_*\cos_*(q|_{A(r, 3r)})} 
        \leq \normc{\cos_*(q|_{A(r,3r)})}.
    \]
    We also have $\normc{q|_{A{(3^ir, 3^{i+1}r)}}} = 2\pi\ln(3) = \normc{q|_{A(r,3r)}}$.
    This implies that 
    \[
    \frac{\normc{\cos_*(q|_{A(3^ir, 3^{i+1}r)})}}{\normc{q|_{A{(3^ir, 3^{i+1}r)}}}} \leq \frac{\normc{\cos_*(q|_{A(r,3r)})}}{\normc{q|_{A(r,3r)}}}
    \]
    and for $k \in \mathbb{N}$,
    \[
    \frac{\normc{\cos_*(q|_{A(r,3^kr)})}}{\normc{q|_{A(r,3^kr)}}} \leq \frac{\normc{\cos_*(q|_{A(r,3r)})}}{\normc{q|_{A(r,3r)}}}.
    \]
    This inequality says that the majority of the mass of the push-forward of $q$ by cosine on $A(r, R)$, where $r \geq \pi$ and $R \geq 3r$, is concentrated in $A(r, 3r)$. We can also assume $r \leq 3\pi$. As in the special case of the annulus $A(\pi, 3\pi)$,  for each such $r \in [\pi, 3\pi]$, we can find $\alpha_r < 1$ such that $\normc{\cos_*(q|_{A(r,3r)})} \leq \alpha_r\normc{q|_{A(r,3r)}}$. Choose $\widetilde{\alpha} := \sup_{r \in [\pi, 3\pi]} \alpha_r$ and note that because $\alpha_r$ depends continuously on $r$ in a compact set, it follows that $\widetilde{\alpha} \in (0,1)$. Therefore, for all $r, R$ satisfying $r \geq \pi$ and $R \geq 3r$, $\normc{\cos_*(q|_{A(r,R)})} \leq \widetilde{\alpha}\normc{q|_{A(r,R)}}$.
\end{proof}

\begin{prop}[Cos-efficient annular push-forwards]\label{cos_eff_annular_pushfowards}
    Let $(q_n)$ be a sequence of measurable quadratic differentials on $\mathbb{C}$. Suppose $(q_n)$ has limit model $\frac{dz^2}{z^2}$ on annuli 
    \[
    A_n := \{z \in \mathbb{C}: r_n < |z-b_n| <R_n \}
    \]
    (where $0<r_n<R_n$ and $b_n \in \mathbb{C}$). If the sequence $(q_n)$ is cos-efficient, then
    \begin{enumerate}[(1)]
        \item there exists a sequence $R_n^* \in (r_n, R_n)$ that satisfies $R_n^* \to 0$ such that $\frac{\norm{q_n}_{A_n^*}}{\norm{q_n}_{A_n}} \to 1$, where $A_n^* := \{z \in \mathbb{C} : r_n < |z-b_n| < R_n^* \}$, and 
        \item if the poles of $(q_n)$ do not converge to critical points of cosine as $n \to \infty$, then $(\cos_*q_n)$ has limit model $\frac{dz^2}{z^2}$ on annuli 
        \[\{z \in \mathbb{C} : |\sin(b_n)|r_n < |z-\cos(b_n)| < |\sin(b_n)|R_n^*\}.
        \]
    \end{enumerate}
\end{prop}

\begin{proof}
   The proof is similar to [\cite{expthursmaps}, Proposition 6.2], adjusted for cosine. We begin with (1). Consider the cos-efficient sequence $(q_n)$. We will write $A_n = A_n' \cup A_n''$ where
   \[
        A_n' := \{z \in A_n : |z-b_n| < \pi \} \text{ and }
        A_n'' := \{z \in A_n : |z-b_n| > \pi \}.
    \]
    
    \underline{Claim 1}:
    We claim that 
    \[\lim_{n \to \infty} \frac{\norm{q_n}_{A_n''}}{\norm{q_n}_{A_n}} = 0.
    \]
    Suppose the claim is not true. Then, $\norm{q_n}_{A_n''}$ accounts for a definition proportion of the mass of $q_n$ on $A_n$. Now 
    \[
    \frac{\normc{\cos_*(q_n|_{A_n})}}{\normc{q_n|_{A_n}}} \leq \frac{\normc{\cos_*(q_n|_{A_n'})}}{\normc{q_n|_{A_n}}} + \frac{\normc{\cos_*(q_n|_{A_n''})}}{\normc{q_n|_{A_n}}}
    \]
    so because $(q_n)$ is cos-efficient, $\normc{\cos_*(q_n|_{A_n'})}$ and $\normc{\cos_*(q_n|_{A_n''})}$ should not lose a definite proportion of the mass each contributes to $\norm{q_n}_{A_n}$. To obtain a contradiction, we will show that this occurs on $A_n''$, by its design. Let $\varepsilon > 0$ be sufficiently small. By the limit model definition, $\lim_{n \to \infty} \norm{q_n}_{A_n} = \beta$ for some $\beta \neq 0$ and $\lim_{n \to \infty} \norm{M_n^*q_n}_{\mathbb{C} \setminus M_n^{-1}(A_n)} = 0$. Combining these facts, there exists $N > 0$ such that if $n > N$,
    \[
    \normc{M_n^*(q_n|_{A_n''}) - c_n\frac{dz^2}{z^2}\Big|_{M_n^{-1}(A_n'')}} < \varepsilon\cdot\beta \quad \text{and} \quad \abs{\norm{q_n}_{A_n} - \beta} < \varepsilon \cdot \beta. 
    \] 
    Using $\abs{\norm{q_n}_{A_n} - \beta} < \varepsilon \cdot \beta$, we have $\beta(1-\varepsilon) < \norm{q_n}_{A_n}$. With this estimate and the estimate $\normc{q_n|_{A_n''} - c_n(M_n)_*\left(\frac{dz^2}{z^2}|_{M_n^{-1}(A_n'')}\right)} < \varepsilon\cdot\beta$, it follows that
    \[
    \normc{c_n(M_n)_*\left(\frac{dz^2}{z^2}|_{M_n^{-1}(A_n'')}\right)} < \varepsilon\cdot\beta + \norm{q_n}_{A_n}.
    \]
    Along with Lemma \ref{cos-efficient-annular-dist-mass}, we have 
    \begin{align*}
    \frac{\normc{\cos_*(q_n|_{A_n''})}}{\normc{q_n|_{A_n}}}
        &= \frac{\normc{\cos_*\left(q_n|_{A_n''} - c_n(M_n)_*\left(\frac{dz^2}{z^2}|_{M_n^{-1}(A_n'')}\right) + c_n(M_n)_*\left(\frac{dz^2}{z^2}|_{M_n^{-1}(A_n'')}\right) \right)}}{\norm{q_n}_{A_n}} \\
        &\leq \frac{\normc{q_n|_{A_n''} - c_n(M_n)_*\left(\frac{dz^2}{z^2}|_{M_n^{-1}(A_n'')}\right)}}{\norm{q_n}_{A_n}} + \frac{\normc{\cos_*\left(c_n(M_n)_*\left(\frac{dz^2}{z^2}|_{M_n^{-1}(A_n'')}\right) \right)}}{\norm{q_n}_{A_n}} \\
        &< \frac{\varepsilon \cdot \beta}{\beta(1-\varepsilon)} + \frac{\normc{\cos_*\left((M_n)_*\left(\frac{dz^2}{z^2}|_{M_n^{-1}(A_n'')}\right) \right)}}{\normc{(M_n)_*\left(\frac{dz^2}{z^2}|_{M_n^{-1}(A_n'')} \right)}} \cdot \frac{\normc{c_n(M_n)_*\left(\frac{dz^2}{z^2}|_{M_n^{-1}(A_n'')} \right)}}{\norm{q_n}_{A_n}} \\
        &< \frac{\varepsilon}{1-\varepsilon} + \widetilde{\alpha}\left(\frac{\varepsilon}{1-\varepsilon} + 1 \right) \\
        &= \frac{\widetilde{\alpha} +\varepsilon}{1-\varepsilon}.
    \end{align*}
    This is less than 1 contradicting the cos-efficiency of $(q_n)$. Thus, $\lim_{n \to \infty} \frac{\norm{q_n}_{A_n''}}{\norm{q_n}_{A_n}} = 0$. Note that this claim implies that $r_n \to 0$, $R_n$ does not tend to infinity, and $R_n < \pi$.
    
    \underline{Claim 2:} We claim that for any fixed $R^* \in (0,\pi)$, 
    \[
    \lim_{n \to \infty} \frac{\norm{q_n}_{M_n(A(R^*,\pi))}}{\norm{q_n}_{A_n}} = 0.
    \]
    So take such an $R^* \in (0, \pi)$ and consider $r_n < R^*$. Let $\varepsilon > 0$ be sufficiently small and let $N = \frac{1}{\varepsilon}$. A key observation is that 
    \[
    \frac{\text{mod}(A(R^*,\pi))}{\text{mod}(A_n)} = \frac{\norm{\frac{dz^2}{z^2}}_{A(R^*,\pi)}}{\norm{\frac{dz^2}{z^2}}_{A(r_n,R_n)}} = \frac{\log\left(\frac{\pi}{R^*} \right)}{\log \left(\frac{R_n}{r_n}\right)} \to 0
    \]
    as $n \to \infty$ because $\frac{R_n}{r_n} \to \infty$ as specified by the limit model. So we may assume $\text{mod}(A_n) > N$. Using this and estimates as in the proof of the previous claim, it follows that for $n > N$, 
    \begin{align*}
        \frac{\norm{M_n^*q_n}_{A(R^*, \pi)}}{\norm{q_n}_{A_n}} &\leq \frac{\normc{q_n|_{M_n(A(R^*,\pi))} - c_n(M_n)_*\left(\frac{dz^2}{z^2}\big|_{A(R^*,\pi)}\right)}}{\norm{q_n}_{A_n}}
        + \frac{\norm{c_n (M_n)_*\left(\frac{dz^2}{z^2}\right)}_{M_n(A(R^*, \pi))}}{\norm{q_n}_{A_n}} \\
        & < \frac{\varepsilon}{1-\varepsilon} + \frac{\norm{c_n(M_n)_*\left( \frac{dz^2}{z^2} \right)}_{M_n(A(R^*, \pi))}}{\norm{c_n(M_n)_*\left( \frac{dz^2}{z^2} \right)}_{A_n}} \cdot \frac{\norm{c_n(M_n)_*\left( \frac{dz^2}{z^2} \right)}_{A_n}}{\norm{q_n}_{A_n}} \\
        &<\frac{\varepsilon}{1-\varepsilon} + \frac{\text{mod}(A(R^*,\pi))}{\text{mod}(A_n)}\left(1+\frac{\varepsilon}{1-\varepsilon}\right) \\
        &< \frac{\varepsilon}{1-\varepsilon} + \varepsilon \cdot \text{mod}(A(R^*,\pi)) \cdot \frac{1}{1-\varepsilon}
    \end{align*}
    and the claim follows. 
    
    Claim 1 and Claim 2 therefore imply that 
    \[
    \lim_{n \to \infty} \frac{\norm{q_n}_{M_n(A(r_n, R^*))}}{\norm{q_n}_{A_n}} = 1.
    \]
    Furthermore, there exists a sequence $R_n^* \in (r_n, R_n)$ with $R_n^* \to 0$ such that 
    \[
    \lim_{n \to \infty} \frac{\norm{q_n}_{A_n^*}}{\norm{q_n}_{A_n}} = 1
    \]
    where $A_n^* := \{z \in \mathbb{C} : r_n < |z-b_n| < R_n^*\}$. This concludes the proof of (1). 
    
    We next prove (2). By (1), it follows that $(q_n)$ has limit model $\frac{dz^2}{z^2}$ on annuli $A_n^*$. Therefore, the affine maps $\tilde{M}_n(z) := z+b_n$ satisfy
    \[
    \lim_{n \to \infty} \normc{(\tilde{M}_n)^*q_n - c_n \frac{dz^2}{z^2}\Big|_{\{r_n < |z| < R_n^*\}}} = 0, 
    \]
    where by our hypothesis, the $b_n$ do not converge to any critical points of cosine. Let $\tilde{\tilde{M}}_n(z) := R_n^*z$ and observe that the limit model definition can also be thought of as follows:
    \begin{align*}
        \normc{(\tilde{M}_n)^*q_n - c_n \frac{dz^2}{z^2}\Big|_{\{r_n < |z| < R_n^*\}}} 
        &= \normc{(\tilde{\tilde{M}}_n)^*(\tilde{M}_n)^*q_n - c_n (\tilde{\tilde{M}}_n)^*\left(\frac{dz^2}{z^2}\Big|_{\{r_n < |z| < R_n^*\}}\right)} \\
        &= \normc{(\tilde{M}_n \circ \tilde{\tilde{M}}_n)^*q_n - c_n \frac{dz^2}{z^2}\Big|_{\{r_n/R_n^* < |z| < 1\}}}.
    \end{align*}
    Define $M_n := \tilde{M}_n \circ \tilde{\tilde{M}}_n$ and define $\hat{M}_n(z) = -R_n^*\sin(b_n)z + \cos(b_n)$. Then, proceeding as in Proposition \ref{cos-eff-int-pushforwards}, let $S_n := (\hat{M_n})^{-1} \circ \cos \circ M_n$ and in the same way, $S_n(z) \to z$ as $n \to \infty$ uniformly on compact subsets of $\mathbb{C}$. Furthermore, 
    \begin{align*}
        &\quad\normc{(\hat{M_n})^*(\cos_*q_n) - c_n\frac{dz^2}{z^2}\Big|_{\{r_n/R_n^* < |z| < 1 \}}} \\
        &\leq \normc{M_n^*q_n - c_n\frac{dz^2}{z^2}\Big|_{\{r_n/R_n^* < |z| < 1 \}}} 
        + \normc{(S_n)_*\left( c_n\frac{dz^2}{z^2}\Big|_{\{r_n/R_n^* < |z| < 1 \}} \right) - c_n\frac{dz^2}{z^2}\Big|_{\{r_n/R_n^* < |z| < 1 \}}} 
    \end{align*}
    and it therefore follows that 
    \[
    \lim_{n \to \infty} \normc{\hat{M}_n^*(\cos_*q_n) - c_n\frac{dz^2}{z^2}\Big|_{\{r_n/R_n^* < |z| < 1 \}}} = 0. 
    \]
    To obtain annuli as in the statement of the proposition, define $\hat{\hat{M}}_n(z) := z + \cos(b_n)$ and  define $\hat{\tilde{M}}_n(z) := -R_n^*\sin(b_n)z$ so that $\hat{M}_n = \hat{\hat{M}}_n \circ \hat{\tilde{M}}_n$, and observe
    \[
        \quad\normc{\hat{M}_n^*(\cos_*q_n) - c_n\frac{dz^2}{z^2}\Big|_{\{r_n/R_n^* < |z| < 1 \}}} 
        = \normc{\hat{\hat{M}}_n^*(\cos_*q_n) - c_n\frac{dz^2}{z^2}\Big|_{\{|\sin(b_n)|r_n < |z| < |\sin(b_n)|R_n^* \}}}.
    \]
    \end{proof}

\subsection{Combining thick and thin cases to produce special annuli}

The majority of the work in \cite{expthursmaps} (that is focused on the characterization of exponential maps) is done to prove a key proposition that essentially says, under appropriate hypotheses, an exp-efficient sequence of quadratic differentials gives rise to special annuli that can be modified to extract the curves in a degenerate Levy cycle. This is exactly what we adapt in this section for cosine. The results below are analogs of Lemma 6.3 and Proposition 3.2 from \cite{expthursmaps}, respectively. We do not prove the former as its proof is unchanged. The latter has a proof that is similar in spirit. 

\begin{lemma}[Sequence of cos-efficient push-forwards]\label{seq_of_cos_eff_push_forwards}
    Let $\lambda_{n,1}, \ldots, \lambda_{n,m} \in \mathbb{C}^*$ where $n, m \in \mathbb{N}$. Suppose that $(q_n)$ is a sequence of measurable quadratic differentials with $0 < \normc{q_n} < \infty$ for all $n$ such that as $n \to \infty$, 
    \[
    \frac{\normc{(C_{\lambda_{n,m}} \circ \cdots \circ C_{\lambda_{n,1}})_*(q_n)}}{\normc{q_n}} \to 1. 
    \]
    Suppose $V_n$ are domains in $\mathbb{C}$ such that for all $n$, there exists $\alpha > 0$ such that 
    \[
    \frac{\norm{q_n}_{V_n}}{\normc{q_n}} \geq \alpha.  
    \]
    Then, for every $1 \leq s \leq m-1$, the sequence $((C_{\lambda_{n,s}} \circ \cdots \circ C_{\lambda_{n,1}})_*(q_n|_{V_n}))$ is cos-efficient. 
\end{lemma}

\begin{proof}[Proof of Proposition \ref{key_prop}]
    This proof follows similarly to [\cite{expthursmaps}, Proposition 3.2]. Suppose (for contradiction) there exists a number of poles $N$, a number of iterates $m$, and a nonzero modulus $M$, such that for all $r \in (0,1)$ the property stated in the lemma does not hold. Therefore, there exists an integrable meromorphic quadratic differential $q$ with at most $N$ poles and there exist $\lambda_1, \ldots, \lambda_m \in \mathbb{C}^*$ so that $\normc{(C_{\lambda_m} \circ \cdots \circ C_{\lambda_1})_*q} > r\normc{q}$ and for all concentric disks $\tilde{D} \subset D$, one of (1), (2), and (3) should not hold. 

    By assumption, choose a sequence $(r_n)$ with $r_n \to 1$ as $n \to \infty$ and choose a corresponding sequence $(q_n)$ of integrable meromorphic quadratic differentials with at most $N$ poles on $\hat{\mathbb{C}}$. There exists corresponding $\lambda_{n,1}, \ldots, \lambda_{n,m} \in \mathbb{C}^*$ so that $\normc{(C_{\lambda_{n, m}} \circ \cdots \circ C_{\lambda_{n, 1}})_*q_n} > r_n\normc{q_n}$. To obtain a contradiction, it suffices to show that there exists a subsequence of $(q_n)$ from which one can extract concentric disks $\tilde{D}_n \subset D_n$ where all of (1), (2), and (3) hold.

    To begin, the Decomposition of Mass Theorem [\cite{expthursmaps}, Theorem 5.2] gives us existence of a subsequence called $(q_n)$ for ease of notation, $\ell \in \mathbb{N}$, and disjoint regions $V_n^{[j]} \subset \hat{\mathbb{C}}$ for $1 \leq j \leq \ell$ with one of the two possible limit models on these regions. Choose some $j$ and for $1 \leq i \leq m$, set
    \begin{align*}
        \hat{q}_n^{(0)} &:= q_n|_{V_n^{[j]}} \\
        \hat{q}_n^{(i)} &:= (C_{\lambda_{n,i}})_*(\hat{q}_n^{(i-1)})
    \end{align*}
    and note that Lemma \ref{seq_of_cos_eff_push_forwards} implies that the $(\hat{q}_n^{(i)})$ are cos-efficient. 

    We now break down into the two cases of the Decomposition of Mass Theorem [\cite{expthursmaps}, Theorem 5.2]. For the thick case, we suppose $(\hat{q}_n^{(0)})$ has limit model $\hat{q}$, where $\hat{q}$ is a meromorphic quadratic differential on $\hat{\mathbb{C}}$ with at most $N$ poles and $0 < \normc{\hat{q}} \leq 1$. We will compute the scalings for the $\hat{q}_n^{(i)}$ inductively. To begin, we know that there exists conformal automorphisms $M_n^{(0)} = a_n^{(0)}z+b_n^{(0)}$ (with $a_n^{(0)}, b_n^{(0)} \in \mathbb{C}$) such that $\lim_{n \to \infty} \normc{(M_n^{(0)})^*\hat{q}_n^{(0)} - \hat{q}} = 0$. Proposition \ref{cos-eff-int-pushforwards} implies that $a_n^{(0)} \to 0$ as $n \to \infty$. Choose $\tilde{R}$ so that $D_{\tilde{R}}(0)$ contains all the poles of $\hat{q}$ in $\mathbb{C}$. Then, the Decomposition of Mass Theorem [\cite{expthursmaps}, Theorem 5.2] guarantees that for each pole $w$ of $\hat{q}$, there exists a $w_n$ of $q_n$ such that $(M_n^{(0)})^{-1}(w_n) \to w$ as $n \to \infty$. So for large $n$, there are at least two poles of $\hat{q}_n^{(0)}$ in $M_n^{(0)}(D_{\tilde{R}}(0))$. Define
    \[
    D_n := M_n^{(0)}(D_{\tilde{R}e^{2\pi M}}(0)) \quad \text{and} \quad \tilde{D}_n := M_n^{(0)}(D_{\tilde{R}}(0)). 
    \]
    Then, $\tilde{D}_n$ contains at least two poles of $\hat{q}_n^{(0)}$ and the modulus of $D_n \setminus \tilde{D}_n$ is at least $M$, showing (1) and (2) hold. Now suppose $k\pi \notin C_{\lambda_{n,s}} \circ \cdots \circ C_{\lambda_{n,1}}(D_n)$ for all $k \in \mathbb{Z}$ and $0 \leq s \leq m$. This implies that the poles of $\hat{q}_n^{(0)}$ do not converge to $k\pi$ for $k \in \mathbb{Z}$ and thus, Proposition \ref{cos-eff-int-pushforwards} implies that $(\cos_*\hat{q}_n^{(0)})$ has limit model $\hat{q}$ with scalings $\hat{M}_n^{(0)}$. Let $M_n^{(1)}(z) := \lambda_{n,1}\hat{M}_n^{(0)}(z)$. Then,
    \begin{align*}
        \lim_{n \to \infty} \normc{(M_n^{(1)})^*\hat{q}_n^{(1)} - \hat{q}} &= \lim_{n \to \infty} \normc{((M_n^{(1)})^{-1})_*(C_{\lambda_{n,1}})_*\hat{q}_n^{(0)} - \hat{q}} \\
        &= \lim_{n \to \infty} \normc{(\hat{M}_n^{(0)})^*(\cos_*\hat{q}_n^{(0)}) - \hat{q}} \\
        &=0
    \end{align*}
    so that $(\hat{q}_n^{(1)})$ has limit model $\hat{q}$ with scalings $M_n^{(1)}(z) = a_n^{(1)}z+b_n^{(1)}$. Similarly, the poles of $\hat{q}_n^{(1)}$ do not converge to $k\pi$ for $k \in \mathbb{Z}$. Continue this process inductively by applying Proposition \ref{cos-eff-int-pushforwards} to obtain new scalings $\hat{M}_n^{(i)}$ for $(\cos_*\hat{q}_n^{(i)})$, defining $M_n^{(i+1)}(z) := \lambda_{n, i+1}\hat{M}_n^{(i)}(z)$, and carrying out the computation above. For $0 \leq i \leq m-1$, we then have that $\hat{q}_n^{(i+1)}$ has limit model $\hat{q}$ with scalings $M_n^{(i+1)}$ and defining $S_n^{(i+1)} := (M_n^{(i+1)})^{-1} \circ C_{\lambda_{n,i+1}} \circ M_n^{(i)}$, the proof of Proposition \ref{cos-eff-int-pushforwards} then gives us that $S_n^{(i+1)}(z) = (\hat{M}_n^{(i)})^{-1} \circ \cos \circ M_n^{(i)}(z) \to z$ uniformly on compact subsets of $\mathbb{C}$. In particular, this is true when $i=m-1$ so it follows that for any $R>0$, $C_{\lambda_{n, m}} \circ \cdots \circ C_{\lambda_{n, 1}}$ is injective on $M_n^{(0)}(D_R(0))$. None of (1), (2), and (3) fail yielding the contradiction for the thick case.

    For the thin case, suppose $\hat{q}_n^{(0)}$ does not have integrable limit model. By the proof of Proposition \ref{cos_eff_annular_pushfowards}, it suffices to assume that our limit model $\frac{dz^2}{z^2}$ is on annuli 
    \[A_n^{(0)} := \left\{z \in \mathbb{C} : \frac{r_n}{a_n^{(0)}} < |z-b_n^{(0)}| < 1 \right\}
    \]
    where $0 < r_n < a_n^{(0)}$, $b_n^{(0)} \in \mathbb{C}$, and $a_n^{(0)} \to 0$. The scaling here is $M_n^{(0)}(z) = a_n^{(0)}z + b_n^{(0)}$. Choose
    \[
    D_n := M_n^{(0)}(D_1(0)) \quad \text{and} \quad \tilde{D}_n := M_n^{(0)}(D_{r_n/a_n^{(0)}}(0)). 
    \]
    By the Decomposition of Mass Theorem [\cite{expthursmaps}, Theorem 5.2], there are at least two poles of $\hat{q}_n^{(0)}$ in the bounded component of $D_n \setminus \tilde{D}_n$ for large $n$. Therefore, (1) and (2) hold. For (3), suppose $k\pi \notin C_{\lambda_{n,s}} \circ \cdots \circ C_{\lambda_{n,1}}(D_n)$ for all $k \in \mathbb{Z}$ and $0 \leq s \leq m$. The scalings for $M_n^{(i)}$ corresponding to $\hat{q}_n^{(i)}$ for $1 \leq i \leq m$ are computed inductively with the idea being very similar to the thick case. The hypothesis implies that the poles of $\hat{q}_n^{(0)}$ do not converge to $k\pi$ for $k \in \mathbb{Z}$ and from Proposition \ref{cos_eff_annular_pushfowards}, it follows that $(\cos_*\hat{q}_n^{(0)})$ has limit model $\frac{dz^2}{z^2}$ with scalings $\hat{M}_n^{(0)}$. Let $M_n^{(1)}(z) := \lambda_{n,1}\hat{M}_n^{(0)}(z)$. Then,
    \begin{align*}
        \lim_{n \to \infty} \normc{(M_n^{(1)})^*\hat{q}_n^{(1)} - \frac{dz^2}{z^2}\Big|_{\{r_n/a_n^{(0)} < |z| < 1 \}} } 
        &= \lim_{n \to \infty} \normc{((M_n^{(1)})^{-1})_*(C_{\lambda_{n,1}})_*\hat{q}_n^{(0)} - \frac{dz^2}{z^2}\Big|_{\{r_n/a_n^{(0)} < |z| < 1 \}} } \\
        &= \lim_{n \to \infty} \normc{(\hat{M}_n^{(0)})^*(\cos_*\hat{q}_n^{(0)}) - \frac{dz^2}{z^2}\Big|_{\{r_n/a_n^{(0)} < |z| < 1 \}} } \\
        &=0
    \end{align*}
    so that $(\hat{q}_n^{(1)})$ has limit model $\hat{q}$ with scalings $M_n^{(1)} = a_n^{(1)}z+b_n^{(1)}$. By hypothesis, the poles of $\hat{q}_n^{(1)}$ do not converge to $k\pi$ for $k \in \mathbb{Z}$ and similar to the thick case, continue this process inductively by applying Proposition \ref{cos_eff_annular_pushfowards} to obtain new scalings $\hat{M}_n^{(i)}$ for $(\cos_*\hat{q}_n^{i})$, defining $M_n^{(i+1)}(z) := \lambda_{n, i+1}\hat{M}_n^{(i)}(z)$, and carrying out the computation above. For $0 \leq i \leq m-1$, it follows that $\hat{q}_n^{(i+1)}$ has limit model $\frac{dz^2}{z^2}$ with scalings $M_n^{(i+1)}$. Define $S_n^{(i+1)}$ analogously to the thick case, and note that the same logic carries through to show that for any $R>0$, $C_{\lambda_{n, m}} \circ \cdots \circ C_{\lambda_{n, 1}}$ is injective on $M_n^{(0)}(D_R(0))$ and hence on $D_n$. Furthermore, (1), (2), and (3) all hold for large $n$ providing the contradiction in the thin case. 
\end{proof}

\section{Generalizing to Other Transcendental Maps and Addressing Critical Points}
\label{generalizing_results_further_section}

It is noted in \cite{expthursmaps} that their results can be possibly used to prove characterizations for other transcendental maps. They also note that the results in Section 6 of their paper are the statements that really need to be generalized. This is exactly what has been done for $C_\lambda$ away from critical points (in Section \ref{adaptations_to_cosine_section}), but should work for other transcendental maps. For more general transcendental maps, the idea of efficiency must hold: efficient sequences of quadratic differentials must have their mass concentrated in very small disks that push forward with almost no loss of mass. 

For $C_\lambda$, the specific properties that were used were the trigonometric identities (such as in Proposition \ref{cos-eff-int-pushforwards} and Lemma \ref{cos-efficient-annular-dist-mass}). The concrete properties of cosine were crucial in understanding how the mass of a quadratic differential was distributed under pushing forward by cosine. For other entire transcendental maps, it may be possible to understand this by either using concrete properties the map may have or if not, by working locally (as efficient quadratic differentials should have their mass concentrated in small disks). In the general setting of entire transcendental maps and away from critical points, to use these techniques, one must be able to achieve the diagram in Figure \ref{change_of_coords_fig}. This diagram encodes the scalings of the thick case (Proposition \ref{cos-eff-int-pushforwards}) and the thin case (proof of Proposition \ref{cos_eff_annular_pushfowards})  coming from the limit models of sequences of quadratic differentials. 

On the other hand, the study of efficiency near the critical points of $C_\lambda$ is still yet to be understood. Due to examples such as Example \ref{example_eff_qd_with_cp}, it may be possible for mass concentrated in small disks near a critical point to push forward with very little loss of mass when there are poles of quadratic differentials that are cos-symmetric (or very close to being cos-symmetric). For example, near $0$, this gap comes from the property $C_\lambda$ has of being an even function. Understanding this will provide a way to either eliminate the \hyperref[condition:mass_condition]{mass condition} or, if not, weaken it in specific situations (for example, some configuration of poles being cos-symmetric). Remark \ref{some_q_do_not_occur} addresses some of these situations in the dynamical setting. Overall, understanding efficiency near the critical points of cosine may provide insight into other families of maps that are not locally affine in some neighbourhood of $\mathbb{C}$.

\bibliographystyle{alpha}

\bibliography{biblio}
\end{document}